\documentclass[12pt]{amsart}

\usepackage{amsthm,amsmath,amssymb,enumerate,color, tikz}
\usepackage{mathtools}
\usepackage{booktabs}

\newcommand{\starprod}{\circledast{}}
\newcommand{\combprod}{\rhd{}}

\DeclareMathOperator{\resolv}{\mathfrak{R}}
\DeclareMathOperator{\spec}{\mathrm{spec}}
\DeclareMathOperator{\dist}{\mathrm{dist}}

\usetikzlibrary{shapes.geometric}

\usepackage{svn}
\SVN $Revision: 118 $

\usepackage[margin=1.5cm]{geometry}
\usepackage[colorlinks,anchorcolor=black,citecolor=black,linkcolor=black]{hyperref}
\usepackage{tikz}

\numberwithin{equation}{section}
\newtheorem{theorem}{Theorem}[section]

\newtheorem{corollary}[theorem]{Corollary}
\newtheorem{lemma}[theorem]{Lemma}
\newtheorem{proposition}[theorem]{Proposition}

\theoremstyle{definition}
\newtheorem{remark}[theorem]{Remark}
\newtheorem{definition}[theorem]{Definition}

\newtheorem{example}[theorem]{Example}

\def\C{\mathbb{C}}       \def\R{\mathbb R}        \newcommand\IR{\mathbb R}        \def\N{\mathbb{N}}       \def\Z{\mathbb{Z}}        \def\Tr{{\rm Tr}}            \def\dim{{\rm dim}}       \def\SG{\mathfrak{S}}    \newcommand\alg[1]{\mathcal{#1}}       
\newcommand\hilbH{\mathcal{H}}   \newcommand\vac{\xi}                    \newcommand{\norm}[1]{\lVert#1\rVert}  
\newcommand{\abs}[1]{
  \left\lvert
    #1
  \right\rvert}
\newcommand\bub[1]{\mathring{#1}}

\DeclareMathOperator\OP{\mathcal{OP}}   \DeclareMathOperator\SP{\mathcal{P}}      \DeclareMathOperator\Int{\mathcal{I}}       \DeclareMathOperator\CI{\mathcal{CI}}

\DeclareMathOperator\Cauchy{\mathfrak{g}}     \DeclareMathOperator\RC{\tilde{\mathfrak{g}}}   \DeclareMathOperator\MG{\mathfrak{m}}     \DeclareMathOperator\HH{\mathfrak{h}}      \DeclareMathOperator\CB{\mathfrak{c}}      \DeclareMathOperator\CBC{c}                    \DeclareMathOperator\KK{\mathfrak{l}}

\newcommand\bG{\mathbf{\Gamma}}
\newcommand\bH{\mathbf{H}}
\newcommand\bS{\mathbf{S}}    \newcommand\bK{\mathbf{K}}    \newcommand\bF{\mathbf{F}}

\begin{document}

\title{Cyclic independence: Boolean and monotone}


\author{Octavio Arizmendi}
\address[Octavio Arizmendi]{
  Centro de Investigaci\'{o}n en Matem\'{a}ticas. Calle Jalisco SN. Guanajuato,
Mexico
}
\email{octavius@cimat.mx}
\thanks{O.A.~gratefully acknowledges financial support by the grants Conacyt
A1-S-9764 and SFB TRR 195.}
\author{Takahiro Hasebe}
\address[Takahiro Hasebe]{Department of Mathematics, Hokkaido University, North 10 West 8,
  Kita-Ku, Sapporo 060-0810, Japan.}
\email{thasebe@math.sci.hokudai.ac.jp}

\thanks{
T.H.~is supported by JSPS Grant-in-Aid for Early-Career Scientists
19K14546 and for Scientific Research (B) 18H0111. 
This work was supported by JSPS Open Partnership Joint Research Projects
grant no. JPJSBP120209921.
}
\author{Franz Lehner}
\address[Franz Lehner]{Institut f\"ur Diskrete Mathematik,
Technische Universit\"at Graz
Steyrergasse 30, 8010 Graz, Austria }
\email{lehner@math.tugraz.at}

\thanks{\mbox{\includegraphics[height=2em]{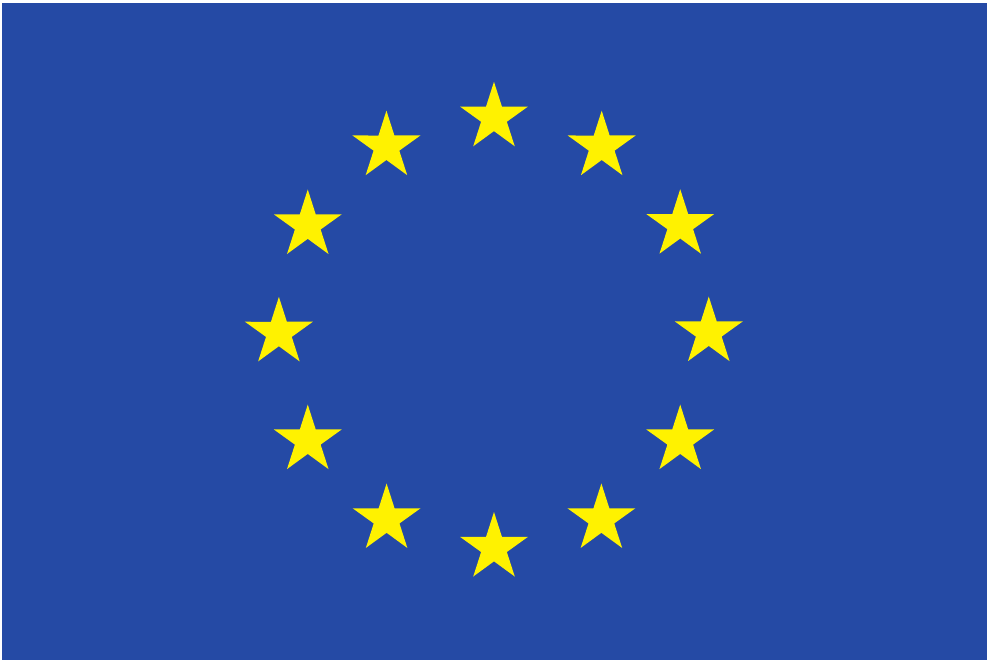}} \parbox[b][1.5em]{3.5cm}{Co-funded
  by\\ the European Union}  F.L. was supported by
  the H2020-MSCA-RISE project 734922 - CONNECT}

\date{Rev.\SVNRevision, \today}

\begin{abstract}
The present paper introduces a modified version of cyclic-monotone independence which originally arose in the context of random matrices, and also introduces its natural analogy called cyclic-Boolean independence. We investigate formulas for convolutions, limit theorems for sums of independent random variables, and also classify infinitely divisible distributions with respect to cyclic-Boolean convolution. Finally, we provide applications to the eigenvalues of the adjacency matrices of iterated star products of graphs and also iterated comb products of graphs. 
\end{abstract}

\maketitle

\section{Introduction}\label{sec1}
The present paper takes its origin in the concept of \emph{cyclic-monotone
independence} which appeared in the study of random matrices 
\cite{CHS18,MR3798863} and which deserves separate treatment; see \cite{MR4260215} for further work.  
The term ``cyclic-monotone independence'' was coined in \cite{CHS18} because of its apparent similarity with monotone independence except that it involves two linear functionals: a state and a tracial linear functional. It abstracts an asymptotic formula for the mixed moments, with respect to the non-normalized trace, of a random rotation of two sets $\mathbf{A}^N$ and $\mathbf{B}^N$ consisting of $N\times N$ deterministic matrices such that all mixed moments of $\mathbf{A}^N$ have finite limits with respect to the \emph{non-normalized trace} as $N$ tends to infinity, and all mixed moments of $\mathbf{B}^N$ have finite limits with respect to the \emph{normalized trace}. More precisely, suppose that $\{A^N_i: 1 \le i \le k\}$ and $\{B^N_i: 1 \le i \le k\}, N=1,2,3,...$ are families of $N\times N$ deterministic matrices that satisfy the following conditions: for any $\ast$-polynomial $P(x_1,x_2,\dots,x_k)$ in non-commuting variables $x_1,x_2,\dots,x_k$ over the field $\C$ without a constant term (e.g.\ $P(x_1,x_2)= x_1^2 x_2 x_1^*$),  the limits 
\[
\lim_{N\to\infty}\Tr_N[P(A_1^N, A_2^N, \dots, A_k^N)] \quad \text{and} \quad \lim_{N\to\infty}\frac{1}{N}\Tr_N[P(B_1^N, B_2^N, \dots, B_k^N)]
\]
exist in $\C$. According to  \cite[Theorem 4.3]{CHS18}, for an $N\times N$ Haar unitary matrix $U^N, N=1,2,3,\dots$, we have the almost sure convergence
\begin{equation}\label{eq:RM}
\begin{split}
&\lim_{N\to\infty}\Tr_N[A_1^N (U^N)^*B_1^N U^N A_2^N (U^N)^*B_2^N U^N \cdots A_k^N (U^N)^*B_k^NU^N]  \\
&\qquad = \lim_{N\to\infty}\Tr_N[A_1^N A_2^N \cdots A_k^N ] \prod_{j=1}^k\left[ \lim_{N\to\infty} \frac{1}{N}\Tr_N[B_j^N]\right]. 
\end{split}
\end{equation}
The formula \eqref{eq:RM} shows some similarity with monotone independence, but
they are not the same because the formula involves both normalized trace and
non-normalized trace.

The present paper offers a simple operator model for cyclic-monotone
independence realized on the tensor product of Hilbert spaces. This
construction also uncovers the associativity of cyclic-monotone independence
with respect to a state and a trace. In order to ensure associativity, we modify the definition of cyclic-monotone independence. The new definition consists of two conditions: one is basically the condition in \cite[Definition 3.2]{CHS18} referring to both the state and the tracial linear functional, and the other is monotone independence with respect to the state (see Definition \ref{def:cyclic_monotone}). 
The modified definition of cyclic-monotone independence shares the same spirit with c-monotone independence \cite{MR2847249} (and c-freeness \cite{BozejkoLeinertSpeicher:1996:convolution}) because they are all associative notions of independence referring to two linear functionals.

However, the relationship between the random matrix model and the operator
model is not perfectly understood. Curiously monotone independence does
not appear in the random matrix model above, although it appears very
naturally in the operator model. This is related to the fact that the random
matrix model above is limited to two families of random matrices $\mathbf{A}^N$
and $\mathbf{B}^N$ and hence the question of associativity is not relevant.

Our operator realization of (modified) cyclic-monotone independence also
indicates that a similar construction works for Boolean independence, which
therefore leads to a notion of \emph{cyclic-Boolean independence}. We develop a
general theory of these two independences: computing generating functions for
the sum of independent random variables, limit theorems, cyclic-Boolean
cumulants which are governed by cyclic-interval partitions and infinitely
divisible distributions with respect to cyclic-Boolean convolution. We do not
know how to define cyclic-monotone cumulants and therefore this question is not
addressed in the present paper.

Moreover, the operator models for cyclic-Boolean and cyclic-monotone independences are directly connected to the star product of (rooted) graphs (see Section~\ref{sec:starproduct})
and the comb product of (rooted) graphs (see Section~\ref{sec:combproduct}). 
Specifically, the eigenvalues of their adjacency matrices can be analyzed by means of cyclic-Boolean independence and cyclic-monotone independence, respectively.  

The techniques are motivated by the relations between the adjacency matrix, the spectrum, the characteristic polynomial and walk generating functions 
of a graph. These form the core subject of 
algebraic graph theory, which deals with various matrices, polynomials and 
generating functions and other invariants carrying information about graphs.

It was shown by Schwenk \cite{Sch74} (later generalized by
Godsil and McKay \cite{GM78})
that the characteristic polynomial
of the star product (or coalescence)
and the comb product (or rooted product)
of graphs
only depends on the characteristic polynomials of the factors and the walk
generating function at the roots of the factors and he gave an explicit formula. 
Similar simple formulas for the generating function of closed walks
starting at the root hold.
While Schwenk's proofs are combinatorial, we will give algebraic proofs based
on the Schur complement which can be generalized to arbitrary matrices and operators. 
Accardi, Ben Ghorbal and Obata \cite{MR2085641} and Obata \cite{MR2062191}
initiated the application of  monotone independence and Boolean independence to the asymptotic spectral analysis of adjacency matrices of iterated comb products and star products, respectively. The operator models for cyclic-monotone and cyclic-Boolean independences extend their work in the sense that the new framework also enables one to analyze refined properties of eigenvalues of the adjacency matrices.   
These generalizations are the subject of the present paper and illustrate
the emergence of new notions of noncommutative independence, i.e., cyclic-monotone and cyclic-Boolean ones. 

To summarize, the main contributions of the present paper are: 
\begin{enumerate}
\item the new notion of cyclic-Boolean independence (Sections \ref{sec3}) and a
 modification of the definition of cyclic-monotone independence given in \cite{CHS18} (Section \ref{sec7}); 
\item  operator models for cyclic-Boolean independence (Sections \ref{sec3}) and for cyclic-monotone independence (Section \ref{sec7}); 
\item convolution formulas for the sum of independent random variables (Sections \ref{sec4}, \ref{sec7}) and their relationships to algebraic graph theory (Section \ref{sec2}); 
\item limit theorems for sums of independent random variables (Sections \ref{sec3}, \ref{sec7}); 
\item cyclic-Boolean cumulants and the relevant partition structure with cyclic-interval partitions (Section \ref{sec5}); 
\item classification of infinitely divisible distributions for cyclic-Boolean convolution (Section \ref{sec6}); 
\item analysis of the asymptotics of the eigenvalues of the adjacency matrices of iterated star product of graphs and iterated comb product of graphs (Sections \ref{sec4}, \ref{sec7}). 
\end{enumerate}

Recently Collins, Leid and Sakuma  found a different matrix model for monotone independence and cyclic monotone independence \cite{CLS}. So far a connection between their model and ours is not clear.

\section{Preliminaries}\label{sec2}
\subsection{Adjacency matrix}
Let $\bG=(V,E)$ be a graph on a vertex set $V=\{v_1,v_2,\dots,v_d\}$
with edge set $E$.
We always consider finite undirected graphs without loops
or multiple edges. An edge between two vertices $u$ and $v$  is denoted by $uv$. The \emph{adjacency matrix} of $\bG$ is the matrix $A=\{a_{ij}\}_{i,j=1}^d$
with entries
\begin{equation*}
  a_{ij} =
  \begin{cases}
    1, & v_iv_j\in E, \\
    0, &\text{otherwise}. 
  \end{cases}
\end{equation*}
The \emph{spectrum of the graph $\bG$} is the spectrum of its adjacency matrix. It consists of the eigenvalues $\lambda_i$ of $A$ which are the roots 
of the \emph{characteristic polynomial} 
\[\phi_\bG(x) =\det(xI-A)=\prod_{i=1}^d(x-\lambda_i).
\]
Alternatively, the eigenvalues of $A$ are the poles of the (tracial)
\emph{Cauchy transform} 
\begin{equation}
  \label{eq:cauchytransform}
  \Cauchy_\bG(z) = \Tr((zI-A)^{-1})
  .
\end{equation}
The Cauchy transform and the characteristic polynomial
are mutually related by the logarithmic derivative
\begin{equation}
  \label{eq:logderiv}
  \Cauchy_\bG(z) = \sum_{i=1}^d \frac{1}{z-\lambda_i}=\frac{d}{dz}\log\phi_\bG(z) = \frac{\phi_\bG'(z)}{\phi_\bG(z)}
  .
\end{equation}
For the generalization of this identity to trace class operators 
it will be convenient to remove the moment of order zero and work with the ``renormalized'' Cauchy transform
\begin{equation}
  \label{eq:modcauchytransform}
  \RC_\bG(z) = \Cauchy_\bG(z) -\frac{d}{z}
  = \Tr((zI-A)^{-1}-z^{-1}I)
\end{equation}
instead.

\subsection{Walk generating functions}
Let $(\bG, o)$ be a finite rooted graph, i.e.,
a graph on vertices $v_1,v_2,\dots,v_d$ where we single
out the vertex $o=v_1$ as the root of the graph. 
The number $m_n$ of closed walks of length $n$ starting at the root $o$ is equal to $\langle A^n e_1, e_1\rangle$
where $A$ is the adjacency matrix of $\bG$ and $e_1$ is the vector $(1,0,0,\dots,0) \in \C^d$. 
Denote by
\begin{equation}
  \label{eq:walkgenfct}
M_\bG(z) = \sum_{n=0}^\infty m_n z^n = \langle (I-zA)^{-1} e_1, e_1\rangle
\end{equation}
the \emph{walk generating function}.
One caution is in place here.
To keep notation simple here and below we do not explicitly
write the root in subscripts,
although the generating functions
depend on the choice of the root.

It will be more convenient to rather work with the resolvent of $A$ and with the \emph{Green function} (evaluated at the root) 
\begin{equation}
  \label{eq:greenfct}
  G_\bG(z) = \langle (zI-A)^{-1} e_1, e_1\rangle = \frac{1}{z}M_\bG\left(\frac{1}{z}\right)
\end{equation}
and its reciprocal
\begin{equation}
  F_\bG(z) = \frac{1}{G_\bG(z)}. 
\end{equation}

We can obtain a relation between the Green function
\eqref{eq:greenfct} and the Cauchy transform  \eqref{eq:cauchytransform} 
from the \emph{Schur complement}.

\subsection{Schur complement}
Let
\begin{equation*}
  M =
  \begin{bmatrix}
    A&B\\
    C&D
  \end{bmatrix}
\end{equation*}
be a block matrix and assume $D$ is invertible. 
Then the \emph{Schur complement} 
\cite {Zhang:2005:schur}
is defined as
\begin{equation}
  \label{eq:defschur}
  M/D = A-BD^{-1}C
  .
\end{equation}
It appears in \emph{Aitken's factorization}
\begin{equation}
  M = 
\begin{bmatrix}
  I & BD^{-1}\\
  0&I
\end{bmatrix}
\begin{bmatrix}
  A-BD^{-1}C&0\\
  0& D
\end{bmatrix}
\begin{bmatrix}
  I & 0\\
  D^{-1}C&I
\end{bmatrix}. 
\end{equation}
which is obtained by Gaussian elimination on the original Matrix $M$.
From this factorization we infer the following assertions:
\begin{enumerate}
 \item  $M$ is invertible if and only
  if $M/D$ is invertible.
  If this is the case, then the \emph{Banachiewicz inversion formula}
  \begin{equation}
    \label{eq:Banachiewicz}
    M^{-1} =
    \begin{bmatrix}
      (A-BD^{-1}C)^{-1} & -(A-BD^{-1}C)^{-1}BD^{-1}\\
      -D^{-1}C(A-BD^{-1}C)^{-1}&   D^{-1}+ D^{-1}C   (A-BD^{-1}C)^{-1} BD^{-1} 
    \end{bmatrix}. 
  \end{equation}
  holds.
 \item  The determinant factorizes and \emph{Jacobi's identity} 
  \begin{equation}
    \label{eq:detM=detM/DdetD}
    \det M = \det(M/D)\det D
  \end{equation}
  holds.
\end{enumerate}

\subsection{Relation between the Green function and the characteristic polynomial}
Let $(\bG,o)$ be a rooted graph on vertices $v_1,v_2,\dots,v_d$.
If we decompose its adjacency matrix
$$
A =
\begin{bmatrix}
  0&b^*\\
  b& D
\end{bmatrix}
$$
into block form with $D=A_{\bG\setminus o}$,
then 
the Green function \eqref{eq:greenfct} is the upper left entry
of the inverse of the matrix
$$
M=
zI-A =
\begin{bmatrix}
  z&-b^*\\
  -b& zI-D
\end{bmatrix}
$$
and coincides with the inverse of its Schur complement,
which results in
\begin{equation*}
  G_\bG(z) = \langle(zI-A)^{-1} e_1,e_1\rangle = (z-b^*(zI-D)^{-1}b)^{-1}
  .
\end{equation*}
Consequently the Schur complement 
\eqref{eq:defschur}
equals $G_\bG(z)^{-1}=F_\bG(z)$,
cf.~\eqref{eq:Banachiewicz}.
Being a matrix of dimension 1 it equals its determinant
and  identity \eqref{eq:detM=detM/DdetD} yields
\begin{equation}
  \label{eq:phiGamma=GGammaphiGamma-o}
  \phi_\bG(x) = F_\bG(x) \,\phi_{\bG\setminus o}(x), \qquad x \in \C\setminus \R.  
\end{equation}

\subsection{Relation between the  Green function and the Cauchy transform of a general matrix}
Let $A$ be a $d\times d$ matrix. We want to understand the relation between
the functions 
$$
G_A(z)  = \langle (z-A)^{-1} e_1, e_1\rangle
\quad
\text{ and }
\quad
\Cauchy_A(z)  = \Tr( (z-A)^{-1} ). 
$$
To this end we partition the matrix into blocks of dimension $1$ and $d-1$
\begin{equation}
  \label{eq:Ablockdecomp}
  A =
  \begin{bmatrix}
    \alpha & a_1^*\\
    a_2 &\bub{A}
  \end{bmatrix}
  .
\end{equation}
Now the corresponding Schur complement   \eqref{eq:defschur} of $z-A$ 
is a scalar
$$
S = z - \alpha - a_1^*(z-\bub{A})^{-1}a_2
$$
and we conclude from the Banachiewicz inversion formula \eqref{eq:Banachiewicz}
that
\begin{equation}
  \label{eq:GA=banach}
  G_A(z) = S^{-1} =
  \frac{1}{z - \alpha - a_1^*(z-\bub{A})^{-1}a_2}
\end{equation}
and
\begin{align*}
 \Cauchy_A(z)  
  &= \Tr
\begin{bmatrix}
    S^{-1} & * \\
    *                                            &
     (z-\bub{A})^{-1} + (z-\bub{A})^{-1}a_2 S^{-1}a_1^*(z-\bub{A})^{-1} 
  \end{bmatrix}
  \\
  &= G_A(z) +  \Cauchy_{\bub{A}}(z) + G_A(z) \Tr( (z-\bub{A})^{-1}a_2a_1^*(z-\bub{A})^{-1} ) \\
  &= G_A(z)(1 + \Tr( a_1^*(z-\bub{A})^{-2} a_2 )) + \Cauchy_{\bub{A}}(z)\\
  &= \frac{F_A'(z)}{F_A(z)} + \Cauchy_{\bub{A}}(z). 
\end{align*}
After subtracting the unit matrix according to 
\eqref{eq:modcauchytransform}
 we obtain the  identity
 \begin{equation}
\label{eq:gAt=gA+DlogzG}
   \RC_{\bub{A}}(z)
   = \RC_A(z) + \frac{d}{dz}\log(zG_A(z))  
   = \RC_A(z) + \frac{1}{z} + \frac{G_A'(z)}{G_A(z)}, 
 \end{equation}
which can be extended to trace class operators.

\subsection{The star product and its adjacency matrix}
\label{sec:starproduct}

Let $(\bG_1,o_1)=(V_1,E_1,o_1)$ and $(\bG_2,o_2)=(V_2,E_2,o_2)$ be rooted
graphs. The star product, denoted by $(\bG_1,o_1)\starprod (\bG_2,o_2)$,
is defined by conglutinating the graphs $(\bG_1,o_1),
(\bG_2,o_2)$ at their roots; see Fig.~\ref{fig:star}.
\begin{figure}
\begin{tikzpicture}
\draw (0,0) -- (1,0) -- (1,1) -- (0,1) -- (0,0); 
\node at (0,0) {$\bullet$};
\node at (1,0) {$\bullet$};
\node at (0,1) {$\bullet$};
\node at (1,1) {$\odot$};
--------
\node at (2,0.5) {$\starprod$}; 
\draw (3,0) -- (4,0) -- (3.5,1) -- (3,0); 
\node at (3,0) {$\bullet$};
\node at (4,0) {$\bullet$};
\node at (3.5,1) {$\odot$};
---------
\node at (5,0.5) {$=$}; 
--------
\draw (6,0.5) -- (6.7,1.2) -- (7.4,0.5) -- (6.7,-0.2) -- (6,0.5); \node at (6,0.5) {$\bullet$};
\node at (6.7,1.2) {$\bullet$};
\node at (7.4,0.5) {$\odot$};
\node at (6.7,-0.2) {$\bullet$};
\node at (6.7,1.2) {$\bullet$};
\draw  (7.4,0.5) -- (8.3,1) -- (8.3,0) -- (7.4,0.5);  \node at (8.3,0) {$\bullet$};
\node at (8.3,1) {$\bullet$};
\end{tikzpicture}
\caption{Star product of rooted graphs}\label{fig:star}
\end{figure}
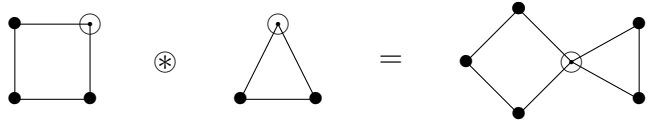
Formally the vertex set of $\bG$ can be realized as a subset of $V_1 \times V_2$, 
\[
V = \{(x_1, o_2): x_1 \in V_1\} \cup  \{(o_1, x_2): x_2 \in V_2\}. 
\]
Two vertices $(x_1,x_2)$ and $(y_1,y_2)$ of $V$ are connected by an edge if
either $x_1 y_1 \in E_1$ and $x_2=y_2=o_2$, or $x_1=y_1=o_1$ and $x_2 y_2 \in
E_2$.
The cartesian product of the vertices corresponds to
the tensor product of the vector spaces and
the adjacency matrix has entries
$$
a_{(x_1,x_2),(y_1,y_2)} = a^{(1)}_{x_1,y_1} \delta_{x_2,o_2}\delta_{y_2,o_2} +
\delta_{x_1,o_1}\delta_{y_1,o_1}  a^{(2)}_{x_2,y_2} 
$$
i.e., 
$$
A = A_1\otimes P_2 + P_1\otimes A_2 
$$
where $P_i$ is the orthogonal projection of $\ell^2(V_{i})$ onto the
one-dimensional subspace spanned by the delta function $\delta_{o_i}$; see
\cite[Proposition 8.50]{MR2316893}.

The star product is associative and hence one may define by iteration
the star product $(\bG,o)=(V,E,o)$ of rooted graphs $(\bG_i,o_i)=(V_i,E_i,o_i),
i=1,2,\dots, N$. Suppose further that those graphs are finite and simple.  Then
the 
vertex set $V$ of $\bG$ can be regarded as a subset of $V_{1} \times
\cdots \times V_{N}$ and hence the adjacency matrix $A_\bG$ can be regarded as
an operator on $\ell^2(V_1) \otimes \cdots \otimes \ell^2(V_N)$. Under this
identification one has  
\begin{equation}\label{eq:boole_sum}
A_\bG = \sum_{i=1}^N P_1 \otimes P_2 \otimes \cdots \otimes P_{i-1} \otimes A_{\bG_i} \otimes P_{i+1} \otimes \cdots \otimes P_N,  
\end{equation} 
Let $\varphi$ be the vacuum state on $B(\ell^2(V))$: $\varphi(X) = \langle X \delta_o, \delta_o\rangle_{\ell^2(V)}$. 

\subsection{The comb product and its adjacency matrix}
\label{sec:combproduct}
Given a graph $\bG_1=(V_1,E_1)$ and a rooted graph $(\bG_2,o_2)=(V_2,E_2,o_2)$,
a new graph $\bG = \bG_1 \combprod (\bG_2,o_2)$
is defined by gluing a copy of $\bG_2$
to every vertex of $\bG_1$ at the root $o_2$:
The vertex set $V$ of $\bG$ is $V_1 \times V_2$ and two vertices $(x_1,x_2)$ and $(y_1,y_2)$ are connected by an edge if and only if either $x_1y_1 \in E_1$ and $x_2=y_2=o_2$, or $x_1 = y_1$ and $x_2 y_2 \in E_2$. 
If we further specify a root of $\bG_1$, then the natural root for the comb product is $(o_1,o_2)$, which makes the comb product associative (but non-commutative) in the category of rooted graphs; see Fig.\ \ref{fig:comb}.  
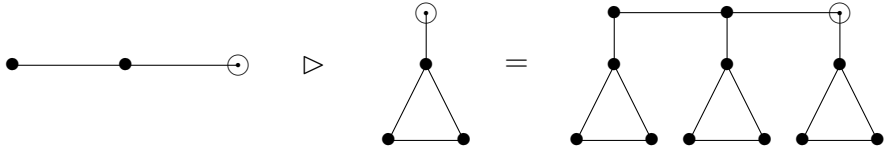
\begin{figure}
\begin{tikzpicture}
\draw (-2,0.5) -- (-0.5,0.5) -- (1,0.5); 
\node at (-0.5,0.5) {$\bullet$};
\node at (-2,0.5) {$\bullet$};
\node at (1,0.5) {$\odot$};
--------
\node at (2,0.5) {$\combprod$}; 
\draw (3,-0.5) -- (4,-0.5) -- (3.5,0.5) -- (3,-0.5); 
\draw (3.5,0.5) -- (3.5,1.2); 
\node at (3,-0.5) {$\bullet$};
\node at (4,-0.5) {$\bullet$};
\node at (3.5,0.5) {$\bullet$};
\node at (3.5,1.2) {$\odot$};
---------
\node at (4.7,0.5) {$=$}; 
--------
\draw (6,1.2) -- (7.5,1.2) -- (9,1.2); 
\draw (6,0.5) -- (6.5,-0.5) -- (5.5,-0.5) -- (6,0.5); 
\draw (6,0.5) -- (6,1.2); 
\draw (7.5,0.5) -- (7,-0.5) -- (8,-0.5) -- (7.5,0.5); 
\draw (7.5,0.5) -- (7.5,1.2); 
\draw (9,0.5) -- (8.5,-0.5) -- (9.5,-0.5) -- (9,0.5); 
\draw (9,0.5) -- (9,1.2); 
\node at (6,1.2) {$\bullet$};
\node at (7.5,1.2) {$\bullet$};
\node at (9,1.2) {$\odot$};
\node at (7.5,1-0.5) {$\bullet$};
\node at (6,1-0.5) {$\bullet$};
\node at (9,1-0.5) {$\bullet$};
\node at (5.5,0-0.5) {$\bullet$};
\node at (6.5,0-0.5) {$\bullet$};
\node at (7,0-0.5) {$\bullet$};
\node at (8,0-0.5) {$\bullet$};
\node at (8.5,0-0.5) {$\bullet$};
\node at (9.5,0-0.5) {$\bullet$};
\end{tikzpicture}
\caption{Comb product of rooted graphs}\label{fig:comb}
\end{figure}
The adjacency matrix now can be written as
$$
A_1\otimes P_2 + I_1\otimes A_2
$$
and 
by iteration one may define the comb product $(\bG,o)$ of a sequence of rooted graphs $(\bG_i,o_i)=(V_i,E_i,o_i)$, $i=1,2,\dots, N$. The adjacency matrix $A_\bG$ can then be regarded as an operator on $\ell^2(V_1) \otimes \cdots \otimes \ell^2(V_N)$ and it has the form
\begin{equation}\label{eq:monotone_sum}
A_\bG = \sum_{i=1}^N I_1 \otimes I_2 \otimes \cdots \otimes I_{i-1} \otimes A_{\bG_i} \otimes P_{i+1} \otimes \cdots \otimes P_N;  
\end{equation} 
see \cite[Proposition 8.38]{MR2316893}. 

For comb products of identical rooted graphs, Accardi, Ben Ghorbal and Obata used  monotone independence satisfied by the summands in \eqref{eq:monotone_sum} in order to study the asymptotics of the Green function of $\bG$ as $N\to \infty$; see the original article \cite[Theorem 5.1]{MR2085641} or the book \cite[Theorem 8.40]{MR2316893}.  On the other hand, for the star product, the summands in \eqref{eq:boole_sum} are Boolean independent, which provides another type of asymptotics of Green function; see the original article of Obata \cite[Theorem 3.7]{MR2062191} or the book \cite[Theorem 8.53]{MR2316893}. 

In the present paper we study the {\bf asymptotic behavior of eigenvalues or
  empirical eigenvalue distributions} of $A_\bG$ for large $N$ using the
asymptotics of the characteristic polynomial $\phi_\bG(z)$ or the Cauchy
transform $\Cauchy_\bG(z)$.

\subsection{Identities for the star product}
For the sake of notational convenience we denote by $\bG_1 \starprod \bG_2$ the star product of two rooted graphs $(\bG_1,o_1)$ and $(\bG_2,o_2)$. The Green function of the star product satisfies 
the following relation, which follows from the decomposition \eqref{eq:boole_sum} of the adjacency matrix into Boolean independent operators and the linearization formula for Boolean convolution in \cite[Section 2]{SW97}; see also \cite{MR2539549} for another proof of the latter. \begin{proposition} For  rooted graphs $(\bG_1,o_1)$ and $(\bG_2,o_2)$ the following formula holds. 
  \begin{align}
    \label{eq:FGamma1.Gamma2}
    F_{\bG_1\starprod\bG_2}(z) = F_{\bG_1}(z) + F_{\bG_2}(z) -z. 
  \end{align}
\end{proposition}

The Cauchy transform of the star product is computed by the formula below. 
\begin{proposition} For  rooted graphs $(\bG_1,o_1)$ and $(\bG_2,o_2)$ the following formula holds. 
  \begin{align}
   \label{eq:tildeg+dlogGGamma1Gamma1}
  \Cauchy_{\bG_1\starprod\bG_2}(z) 
     +\frac{G_{\bG_1\starprod\bG_2}'(z)}{G_{\bG_1\starprod\bG_2}(z)}     
  &= \Cauchy_{\bG_1}(z) + \Cauchy_{\bG_2}(z)
     +\frac{G_{\bG_1}'(z)}{G_{\bG_1}(z)}
     +\frac{G_{\bG_2}'(z)}{G_{\bG_2}(z)}. 
  \end{align}
\end{proposition}
\begin{remark}
Later we will give two alternative proofs in a more general setting; see Theorem \ref{thm:convolution}.
\end{remark}
\begin{proof}
The key is the simple identity 
\begin{equation}
 \label{eq:phiGamma1.Gamma2minuso}
    \phi_{(\bG_1\starprod\bG_2)\setminus o}(x) 
= \phi_{\bG_1\setminus o_1}(x) \, \phi_{\bG_2\setminus o_2}(x), 
\end{equation}
for the star product,
which follows from the fact that the removal of the root splits the graph into two disjoint connected components
  \begin{equation*}
    (\bG_1\starprod\bG_2)\setminus o=(\bG_1\setminus o_1)\cup(\bG_2\setminus o_2).   
  \end{equation*}
  Using the Schur identity \eqref{eq:phiGamma=GGammaphiGamma-o}
  we can rewrite  \eqref{eq:phiGamma1.Gamma2minuso} as
  \begin{equation}\label{eq:star}
    \phi_{\bG_1\starprod\bG_2}(z)\,
    G_{\bG_1\starprod\bG_2}(z)
    = 
    \phi_{\bG_1}(z)\,
    G_{\bG_1}(z)\,
    \phi_{\bG_2}(z)\,
    G_{\bG_2}(z)
  \end{equation}
  and taking the logarithmic derivative of \eqref{eq:star} together with \eqref{eq:logderiv} yields 
   \eqref{eq:tildeg+dlogGGamma1Gamma1}.
\end{proof}

Finally, the characteristic polynomial of the star product satisfies the following 
identity proved by Schwenk, for which we give an alternative proof.

\begin{theorem}[{\cite[Corollary 2b]{Sch74}, \cite[Lemma~9.1]{GM76b}}] For  rooted graphs $(\bG_1,o_1)$ and $(\bG_2,o_2)$ the following formula holds. 
  \begin{equation}
    \phi_{\bG_1\starprod\bG_2}(x) 
     = \phi_{\bG_1}(x)\,
       \phi_{\bG_2\setminus o_2}(x)
       +
       \phi_{\bG_1\setminus o_1}(x) \,
       \phi_{\bG_2}(x)
       - 
       x\phi_{\bG_1\setminus o_1}(x) \,
       \phi_{\bG_2\setminus o_2}(x). 
  \end{equation}
\end{theorem}
\begin{proof}
Equations \eqref{eq:star} and \eqref{eq:FGamma1.Gamma2} give rise to 
 \begin{align*}
 \phi_{\bG_1 \starprod \bG_2}(x)  
 &=  \phi_{\bG_1}(x) \phi_{\bG_2}(x)G_{\bG_1}(x) G_{\bG_2}(x) (F_{\bG_1}(x) + F_{\bG_2}(x)-x) \\
&=   \phi_{\bG_1}(x) \phi_{\bG_2}(x) (G_{\bG_1}(x) + G_{\bG_2}(x)-xG_{\bG_1}(x)G_{\bG_2}(x)). 
 \end{align*}  
Substituting the formula \eqref{eq:phiGamma=GGammaphiGamma-o}, $G_{\bG_i}(x) = \phi_{\bG_i\setminus o}(x)/\phi_{\bG_i}(x)$, into the above yields the desired formula. 
\end{proof}

\subsection{Identities for the comb product}
In this section, the comb product is denoted simply by $\bG_1 \combprod \bG_2$, the root being omitted for simplicity, for a graph $\bG_1$ and a rooted graph $(\bG_2,o_2)$.  The relation between the Green functions is 
simple and follows from the decomposition \eqref{eq:monotone_sum} of the adjacency matrix together with Muraki's formula \cite[Theorem 3.1]{Mur00}; see also \cite[Theorem 3.2]{MR2421124} for another proof of Muraki's formula.  
\begin{proposition} For  rooted graphs $(\bG_1,o_1)$ and $(\bG_2,o_2)$ the following formula holds: 
  \begin{equation}
    \label{eq:FGammaH=FGammaFH}
    F_{\bG_1 \combprod \bG_2}(z)  = F_{\bG_1}(F_{\bG_2}(z)). 
  \end{equation}
\end{proposition}

For the characteristic polynomial 
Schwenk proved the following relation by combinatorial arguments;
we give an algebraic proof based on the simpler relation 
\eqref{eq:FGammaH=FGammaFH}.
\begin{theorem}[{\cite[Theorem 5]{Sch74}}] \label{thm:Schwenk}  Let $\bG$ be a graph on $d$ vertices and $(\bH,o)$ be a  rooted graph. Then 
  \begin{equation}\label{eq:Schwenk}
    \phi_{\bG \combprod \bH}(x) 
    = \phi_{\bH\setminus o}(x)^d \phi_\bG(\phi_\bH(x)/\phi_{\bH\setminus o}(x))
    = \phi_{\bH\setminus o}(x)^d \phi_\bG(F_\bH(x)). 
  \end{equation}
\end{theorem}
\begin{proof}
We proceed by induction.
Fix an arbitrary vertex $o'$ of $\bG$ as a root.
Then removing the root $o''$ from $\bG \combprod \bH$ splits off an extra copy of $\bH\setminus o$ (cf.\ Fig.\ \ref{fig:comb})
$$
(\bG \combprod \bH)\setminus o'' = ((\bG\setminus o') \combprod \bH)\cup (\bH\setminus o)
$$
and therefore
$$
\phi_{(\bG \combprod \bH)\setminus o''}(x)  = \phi_{(\bG\setminus o') \combprod \bH}(x)\,\phi_{\bH\setminus o}(x). 
$$
We proceed with identities \eqref{eq:phiGamma=GGammaphiGamma-o} and 
\eqref{eq:FGammaH=FGammaFH} to conclude by induction
\begin{align*}
  \phi_{\bG \combprod \bH}(x)
  &=\phi_{(\bG \combprod \bH)\setminus o'}(x)
    F_{\bG \combprod \bH}(x)
    \\
   &= \phi_{(\bG\setminus o')\combprod \bH}(x) \,
      \phi_{\bH\setminus o}(x) \,
      F_{\bG \combprod \bH}(x) \\
   &= \phi_{\bH\setminus o}(x)^{d-1}\,
      \phi_{\bG\setminus o'}(F_\bH(x))\,
      \phi_{\bH\setminus o}(x) \,
      F_\bG(F_\bH(x))\\
   &= \phi_{\bH\setminus o}(x)^{d}\,
      \phi_\bG(F_\bH(x)). 
\end{align*}
\end{proof}
Finally, formula \eqref{eq:Schwenk} gives rise to an equivalent formula for the renormalized Cauchy transform. 

\begin{proposition} In the setting of Theorem~\ref{thm:Schwenk}, one has 
\[
\RC_{\bG \combprod \bH}(z)  = d \RC_{\bG} (z) +F_{\bH}'(z) \, \RC_{\bG}(F_{\bH}(z)). 
\]
\end{proposition}
\begin{remark}
 Later we will give  two more proofs in a more general setting; see Theorem \ref{thm:CM_convolution}. 
\end{remark}
\begin{proof}
Combining \eqref{eq:logderiv},  \eqref{eq:Schwenk} and \eqref{eq:phiGamma=GGammaphiGamma-o} yields
\begin{align*}
\Cauchy_{\bG \combprod \bH}(z) 
&= d  \frac{d}{dz}\log \frac{\phi_{\bH}(z)}{F_{\bH} (z)}+F_{\bH}'(z) \, \Cauchy_{\bG}(F_{\bH}(z))  \\
&= d \left(\Cauchy_\bH(z) -\frac{F_\bH'(z)}{F_{\bH}(z)}\right)+F_{\bH}'(z) \, \Cauchy_{\bG}(F_{\bH}(z)),  
\end{align*}
which can be rewritten as
\[
  \RC_{\bG \combprod \bH}(z) + \frac{d d'}{z}
  = d \left( \RC_{\bH} (z)  + \frac{d'}{z} -\frac{F_\bH'(z)}{F_{\bH}(z)}\right) +F_{\bH}'(z)\left(\RC_{\bG}(F_{\bH}(z)) +\frac{d}{F_\bH(z)}\right), 
\]
where $d'$ is the number of vertices of $\bH$. 
\end{proof}

\section{Cyclic-Boolean independence}\label{sec3}

In order to study the eigenvalues of the adjacency matrix of star product graphs, we will compute traces of powers of the adjacency matrix. These computations can be abstracted and formulated as a new notion of independence, which we call cyclic-Boolean independence.

\subsection{Definition and example}

The definition of cyclic-Boolean independence is motivated by the star product
from Section~\ref{sec:starproduct}, which can be extended to the general setting of Hilbert spaces as follows.

\begin{example}\label{exa:CB}
Let $H_i, i\in \N,$ be Hilbert spaces with distinguished unit vectors $\xi_i \in H_i$, $P_i\colon H_i \to H_i$ the orthogonal projection onto $\mathbb C \xi_i$, $T(H_i)$ the $*$-algebra of trace-class operators on $H_i$ and $\varphi_i$ the vector state on $B(H_i)$ defined by $\varphi_i(A) = \langle A \xi_i, \xi_i\rangle$. 

Let $H= H_1 \otimes \cdots \otimes H_N$, $\xi = \xi_1 \otimes \cdots \otimes \xi_N$ and $\varphi$ the vacuum state on $B(H)$ defined by $\xi$. 
Let $\pi_i \colon B(H_i) \to B(H)$ be the $*$-homomorphism defined by 
\begin{equation}
  \label{eq:embedbool}
\pi_i(A) = P_1 \otimes \cdots \otimes P_{i-1} \otimes A \otimes P_{i+1} \otimes \cdots \otimes P_N. 
\end{equation}
The family of $\ast$-subalgebras $\{\pi_i(B(H_i))\}_{i=1}^N$ is Boolean independent in $(B(H), \varphi)$; e.g.\ see \cite[Theorem 8.8]{MR2316893}.  
Furthermore, we compute the mixed moments with respect to the trace.  A key formula is 
\begin{equation}\label{eq:proj}
P_i A P_i = \varphi_i(A)P_i, \qquad A \in B(H_i).  
\end{equation}
For any cyclically alternating tuple $(k_1,\dots, k_n) \in \N^n$, namely those
satisfying $k_1 \neq k_2 \neq \cdots \neq k_n \neq k_1$, and for any $A_i \in
T(H_{k_i})$ a direct computation using formula \eqref{eq:proj} yields
$$
\Tr_{H}(\pi_{k_1}(A_1) \cdots \pi_{k_n} (A_n)) = 
\begin{cases} 
\Tr_{H_{k_1}}(A_1), &n=1, \\ \varphi_{k_1}(A_1)\,\varphi_{k_2}(A_2) \cdots \varphi_{k_n}(A_n), & n \geq2. 
\end{cases}
$$
\end{example}
Let us raise this identity to an abstract concept.

\begin{definition} Let $\mathcal A$ be a $\ast$-algebra over $\C$,  $\varphi$ a positive linear functional on $\mathcal A$ and $\omega$ a positive tracial linear functional on $\mathcal A$. The triplet $(\alg{A},\varphi,\omega)$ is called a \emph{cyclic non-commutative probability space} (cncps). 
The \emph{distribution} of a self-adjoint element $a \in \alg{A}$ is the data $\{(\varphi(a^n), \omega(a^n)): n\ge1\}$. 
\end{definition}
\begin{definition} Let $(\alg{A},\varphi,\omega)$ be a cncps. A family of
  $\ast$-subalgebras $\{\mathcal A_k\}_{k \in K}$ is said to be
  \emph{cyclic-Boolean independent} if 
\begin{enumerate}[\rm(i)] 

\item it is Boolean independent with respect to $\varphi$, that is, for any $n\geq2$, alternating tuple $(k_1,\dots, k_n) \in K^n$ (namely, with $k_1 \neq k_2 \neq \cdots  \neq k_n)$ and $a_i \in \mathcal{A}_{k_i}, i=1,2,\dots,n$, one has the factorization  
$$
\varphi(a_1 \cdots a_n) = \varphi(a_1)\varphi(a_2) \cdots \varphi(a_n);  
$$

\item for any $n\ge1$, any cyclically alternating tuple $(k_1,\dots, k_n) \in K^n$ and any choice of $a_i \in \mathcal{A}_{k_i}, i=1,2,\dots,n$, one has 
$$
\omega(a_1 a_2 \cdots a_n) = 
\begin{cases} 
\omega(a_1), &n=1, \\ 
\varphi(a_1)\varphi(a_2) \cdots \varphi(a_n), & n \geq2. 
\end{cases}
$$
\end{enumerate}
A family of
  elements $\{a_k\}_{k\in K}$ of $\alg{A}$ is said to be cyclic-Boolean
  independent if this is the case for $\{\alg{A}_k\}_{k\in K}$, where $\alg{A}_k$ is the $\ast$-subalgebra generated by $a_k$ without unit. 
\end{definition}

\begin{example}
Suppose that $\{a,b,c\}$ is cyclic-Boolean independent in $(\mathcal{A},\varphi,\omega)$. Then 
\[
\varphi(ba^2bc^2b) = \varphi(b)\varphi(a^2)\varphi(b)\varphi(c^2)\varphi(b)
\]
and 
\[
\omega(ba^2bc^2b) = \varphi(b^2)\varphi(a^2)\varphi(b)\varphi(c^2). 
\]
\end{example}

Another operator model occurs on star products of Hilbert spaces. 
\begin{example}
\label{ex:star_product}
  Let $\hilbH_i$ be separable Hilbert spaces with distinguished unit vectors
  $\xi_i$ as above and $\bub{\hilbH}_i=(\C\xi_i)^\perp$.
  The
\emph{star product} of the Hilbert spaces $\hilbH_i$ is the direct sum
  $$
  \hilbH=\C\xi\oplus \bigoplus_i \bub{\hilbH}_i
  .
  $$
  Then each $\hilbH_i$ can be identified with the subspace
  $\C\xi\oplus\bub{\hilbH}_i\subseteq   \hilbH$
  and there is a canonical representation
  of  $B(\hilbH_i)$ on $\hilbH$ which acts by simply annihilating the 
  complement of $H_i$.
  More precisely we decompose $\hilbH$ as a direct sum
  $\hilbH\simeq\hilbH_i\oplus\hilbH_i^\perp$
  where $\hilbH^\perp_i = \bigoplus_{j\ne i}\bub{\hilbH}_j$ 
  and  define the representation
  $\pi_i(X)=X\oplus 0$.
  Then the algebras $\alg{A}_i=\pi_i(B(\hilbH_i))$ are Boolean independent
  with respect to the vacuum expectation $\varphi=\langle\ldotp \xi,\xi\rangle$
  and
  moreover, the algebras $\alg{A}_i$ are cyclic-Boolean independent
  with respect to the trace.
  Indeed, let $P_0\in B(\hilbH)$  be the projection onto $\C\xi$
  and $P_i$ the projections onto $\bub{\hilbH}_i$; 
  then $P_0,P_1,P_2,\dots$ form a partition of unity
  and by definition we have $X=(P_0+P_i)X(P_0+P_i)$ for all $X \in \alg{A}_i$.
  Let $X_1X_2\dots X_n$ be a cyclically alternating product
  of trace class operators with $X_k\in\alg{A}_{i_k}$, then
  \begin{align*}
    \Tr(X_1X_2\dotsm X_n)
    &=    \Tr((P_0+P_{i_1})X_1(P_0+P_{i_1})(P_0+P_{i_2})X_2(P_0+P_{i_2})\dotsm (P_0+P_{i_n})X_n(P_0+P_{i_n})) \\
    &=    \Tr(P_0X_1P_0X_2P_0\dotsm P_0X_nP_0) \\
    &= \varphi(X_1)\,\varphi(X_2)\dotsm\varphi(X_n). 
  \end{align*}
\end{example}
Next we show that  any Boolean independent family
can be represented on a star product space.

\subsection{Construction of a cyclic-Boolean trace}\label{sec:construction}

Let $(\alg{A},\varphi)$ be a noncommutative probability space,
  where $\alg{A}$ is a $*$-algebra 
  and $\alg{A}_i$ are Boolean independent subalgebras.
  In the following assume that $\alg{A}$ is faithfully represented on a Hilbert space $\hilbH$
  and that the state $\varphi$ is realized as a vector state
  $\varphi(X) = \langle X\xi,\xi \rangle$. One way to achieve this under
  certain conditions is the  \emph{GNS-construction}.

  Recall that the GNS-representation consists of the Hilbert
  space $\hilbH_\varphi$ obtained by completing the quotient space
  $\alg{A}/N_\varphi$, where $N_\varphi=\{x\in\alg{A} \mid \varphi(x^*x)=0\}$,
  with respect to the scalar product
  $$
  \langle [x]_\varphi, [y]_\varphi\rangle = \varphi(y^*x)
  .
  $$
  The action of the GNS representation is $\pi_\varphi(x) [y]_\varphi = [xy]_\varphi$.

\begin{lemma}
    The GNS representation is faithful if and only if the state $\varphi$ is
    nondegenerate in the sense that
    if $\varphi(axb)=0$ for all $a,b\in\alg{A}$, then $x=0$.
  \end{lemma}
  \begin{proof}
    Let $x\in\alg{A}$, then
    $\pi_\varphi(x) =0$ $\iff$ $\pi_\varphi(x)[y]_\varphi=0$ for all $y\in \alg{A}$
    $\iff$ $\langle [xy]_\varphi,[z]_\varphi\rangle_\varphi=0$ for all $y,z\in \alg{A}$
    $\iff$ $\varphi(z^*xy)=0$ for all $y,z\in \alg{A}$.
  \end{proof}
  
  If $\alg{A}$ is unital, then the state vector $\xi=[1]_\varphi$ comes for free,
  otherwise the state must satisfy the Cauchy-Schwarz condition
  $$
  \abs{\varphi(x)}^2\leq C\varphi(x^*x)
  $$
  for some fixed constant $C$ in order to allow a positive extension 
  to the unitization of $\alg{A}$,
  see \cite[Theorem~4.5.11]{Rickart:1960:general}. 

  Assuming that $\alg{A}$ and the state $\varphi$ are faithfully
  represented on some Hilbert space $\hilbH$
  we identify $\alg{A}$ with a subalgebra of $B(\hilbH)$ and
  we are now going to reconstruct the star product space from this data.
Let $\hilbH_0=[\xi]=\C\xi$ be the subspace spanned
  by $\xi$ and $P_0$ the orthogonal projection onto it.
  Adjoining this projection to the algebra $\alg{A}$
  and to each subalgebra $\alg{A}_i$, Boolean independence is preserved
  and wlog we may assume that $P_0\in\alg{A}_i$ for every $i$.
  Let now $\bub{\alg{A}}_i=\ker\varphi\cap\alg{A}_i$, then
  we can construct the components of the star product space as follows.
\begin{lemma}
    Let $\hilbH_0=[\xi]$ and  $\bub{\hilbH}_i=[\bub{\alg{A}}_i\xi]$ 
    the closed invariant subspace generated by $\xi$.
    Then
    \begin{enumerate}[(i)]
     \item $\hilbH_0\perp\hilbH_i$ for all $i$.
     \item $\hilbH_i\perp\hilbH_j$ for all $i\ne j$.
    \end{enumerate}
  \end{lemma}
\begin{proof}
    It suffices to verify orthogonality on the dense subspaces $\bub{\alg{A}}_i\xi$.
    \begin{enumerate}[(i)]
     \item 
      Let $X\in \bub{\alg{A}}_i$, then
      $$
      \langle X\xi, \xi\rangle = \varphi(X)=0. 
      $$
     \item  
      Let $X\in \bub{\alg{A}}_i$ and  $Y\in \bub{\alg{A}}_j$ with $i\ne j$,
      then
      $$
      \langle X\xi,Y\xi\rangle = \varphi(Y^*X)=\varphi(Y^*)\,\varphi(X)=0. 
      $$
    \end{enumerate}
  \end{proof}
We now construct the decomposition.
  Under the assumption that $P_0\in\alg{A}_i$ we have
  $\hilbH_i:=\hilbH_0\oplus\bub{\hilbH}_i=[\alg{A}_i\xi]$.
  Denote by $P_i$ the projection onto $\bub{\hilbH}_i$,
  by $\hat{\alg{A}}$ the subalgebra of $\alg{A}$ generated
  by $(\alg{A}_i)_{i\in I}$
  and let $\hat{\hilbH}=[\hat{\alg{A}}\xi]\subseteq \hilbH$
  the closed invariant subspace generated by $\xi$.
  Let further $\hat{P}$ and  $\hat{P}^\perp$ be the respective projections
  onto the space $\hat{\hilbH}$ and its orthogonal complement $\hat{\hilbH}^\perp$.
\begin{proposition}
    \begin{enumerate}[(i)]
     \item []
     \item  For each $i$ the space $\hilbH_i$
      is invariant under $\alg{A}_i$, i.e., for $X\in\alg{A}_i$
      \begin{equation}
        X(P_0+P_i)  = (P_0+P_i) X (P_0+P_i). 
      \end{equation}
     \item
      \label{it:Hj=kerAi}
      For $i\ne j$ the subspace $\bub{\hilbH}_j$ is annihilated
      by $\alg{A}_i$,
      i.e., for $X\in\alg{A}_i$ 
    \begin{equation}
      XP_j=P_jX=0.   
    \end{equation}
   \item The space $\hat{\hilbH}$ is the closed linear span of the subspaces $\alg{A}_i\xi$,
    i.e.,
    \begin{equation}
      \label{eq:fockspacedecomp}
    \hat{\hilbH} = \hilbH_0\oplus\bigoplus_i\bub{\hilbH}_i. 
    \end{equation}
  \end{enumerate}
  \end{proposition}

\begin{proof}
    \begin{enumerate}[(i)]
     \item []
     \item This is an immediate consequence of the definition.
     \item It suffices to show that $XP_j=0$, i.e., $X$ vanishes on
      $\bub{\hilbH}_j$. We verify this on the dense subspace $\bub{A}_j\xi$.
      Indeed, let $Y\in\bub{\alg{A}}_j$, then
      $$
      \norm{XY\xi}^2 
      =   \langle XY\xi,XY\xi\rangle
      = \varphi(Y^*X^*XY) = \varphi(Y^*)\,\varphi(X^*X)\,\varphi(Y) = 0;
      $$
      finally $P_jX=(X^*P_j)^*=0$ because $\alg{A}_i$ is a $*$-algebra.
     \item 
      The space $\hat{\hilbH}$ is the closure of the span of the alternating words
      $X_1X_2\dots X_n\xi$ with $X_j\in\alg{A}_{i_j}$ and $i_j\ne j_{j+1}$.
      We claim that such a word satisfies
      $X_1X_2\dotsm X_n\xi\in\hilbH_{i_1}$.
      We proceed by induction.
      The claim is obviously true for $n=1$.
      Let now $\eta = X_1X_2\dotsm X_n\xi$ be a given word.
      Then $\eta= X_1\eta'$ with $\eta'=X_2\dotsm X_n\xi$
      and by induction hypothesis $\eta'\in H_{i_2}$,
      say $\eta'=\alpha\xi + \eta''$ with
      $\eta''\in \bub{\hilbH}_{i_2}$.
      But then from item
      \eqref{it:Hj=kerAi} we infer that
      $X_1\eta' = \alpha X_1\xi + 0\in\hilbH_{i_1}$.
    \end{enumerate}
  \end{proof}

\begin{corollary}
Every $X\in\hat{\alg{A}}$ has block decomposition
    \begin{equation}
      \label{eq:X=PXP+PpXPp}
      X = \hat{P}X\hat{P} + \hat{P}^\perp X\hat{P}^\perp
    \end{equation}
    and more precisely every $X\in \alg{A}_i$ has the block decomposition
    \begin{equation}
      X = (P_0+P_i)X(P_0+P_i) + \hat{P}^\perp X\hat{P}^\perp. 
    \end{equation}
\end{corollary}

  \begin{theorem}\label{omega}
    The functional $\omega(X)= \Tr(\hat{P}X\hat{P})$
    is a semifinite trace on the algebra $\hat{\alg{A}}$
and
    the subalgebras $\alg{A}_i \cap T(\hilbH)$ are cyclic-Boolean independent with respect
    to $\omega$.
  \end{theorem}
  \begin{proof}
    $\omega$ is a trace on $\hat{\alg{A}}$
    because $\hat{P}$ is in the commutant of $\hat{\alg{A}}$.
    Now
    let $X_1X_2\dotsm X_n$ 
    be a cyclically alternating product with $X_j\in\alg{A}_{i_j}$ for $j=1,2,\dots,n$,
    then we have
    \begin{align*}
    \omega(X_1X_2\dotsm X_n)
      &=
        \begin{multlined}[t]
        \Tr \bigl(\hat{P}
          ( (P_0+P_{i_1})X_1(P_0+P_{i_1}) + \hat{P}^\perp X_1\hat{P}^\perp)
      ( (P_0+P_{i_2})X_2(P_0+P_{i_2}) + \hat{P}^\perp X_2\hat{P}^\perp)
      \\
      \dotsm
          ( (P_0+P_{i_n})X_n(P_0+P_{i_n}) + \hat{P}^\perp X_n\hat{P}^\perp)
          \hat{P}
          \bigr)
        \end{multlined}
    \\
    &=\Tr \bigl(
           (P_0+P_{i_1})X_1(P_0+P_{i_1}) 
          ( (P_0+P_{i_2})X_2(P_0+P_{i_2})
          \dotsm
          ( (P_0+P_{i_n})X_n(P_0+P_{i_n})
      \bigr)
      \\
    &=\Tr \bigl(
           (P_0+P_{i_1})X_1P_0X_2P_0\dotsm P_0X_n(P_0+P_{i_n})
          \bigr)
      \\
&=\Tr (P_0 X_1P_0 X_2P_0\dotsm P_0X_nP_0)
    \end{align*}
    and it follows that $X_1X_2\dotsm X_n\in L^1(\omega)$.
  \end{proof}

  \begin{remark}
    Conversely, assume that subalgebras $\alg{A}'$
    and $\alg{A}''$ are cyclic-Boolean independent in a cnps
    $(\alg{A},\varphi,\omega)$.
    Assume further that $\alg{A}$ is generated by $\alg{A}'$ and
     $\alg{A}''$ 
    and that there is a projection $p\in\alg{A}'$
    such that $pap=\varphi(a)p$ for $a\in\alg{A}'$ and $\omega(p)=\varphi(p)=1$.
    Then $\varphi(x)=\omega(px)$ for all $x\in\alg{A}$.
  \end{remark}

\section{Convolution and central limit theorem}\label{sec4}

\subsection{Cyclic-Boolean convolution}
Let $(\mathcal A, \varphi, \omega)$ be a cncps.
For  $a \in\alg{A}$ the renormalized (tracial)
Cauchy transform is the formal Laurent series 
\begin{equation*}
\RC_a(z) =  \sum_{n=1}^\infty
\frac{\omega(a^n)}{z^{n+1}}. 
\end{equation*}
By slight
 abuse of terminology, we call $G_a$ the Green function (evaluated
at the state $\varphi$) of $a$. It has formal Laurent expansion
\begin{equation*}
  G_a(z)
  = \frac{1}{z} + \sum_{n=1}^\infty \frac{\varphi(a^n)}{z^{n+1}}, 
\end{equation*}
and we denote the reciprocal Green function by $F_a(z)=1/G_a(z)$.

If $a$ is a trace class operator on a Hilbert space and $\omega$ is the trace
then $|\Tr(a^n)| \leq \norm{|a|^{n-1}}\, \Tr(|a|) \leq  \|a\|^{n-1} \Tr(|a|)$,
and hence $\RC_a(z)$ is absolutely convergent in $\{z\in \C:
|z|>\|a\|\}$. Moreover, if $a$ is selfadjoint then $\RC_a$ has analytic
extension to $\C \setminus \spec(a)$ by Lidskii's theorem
\begin{equation}
\RC_a(z) = \Tr((z-a)^{-1}-z^{-1}) = \sum_{i=1}^\infty \frac{\lambda_i}{z(z -\lambda_i)},  
\end{equation}
where $\{\lambda_i\}_{i\geq1}$ is the multiset of eigenvalues of $a$. In particular, the non-zero eigenvalues of $a$ can be detected from $\RC_a$ as poles.  If the Hilbert space is finite-dimensional, then we also have the formula 
\begin{equation}
\RC_a(z) =  \sum_{i=1}^{\dim (H)} \frac{1}{z -\lambda_i}- \frac{\dim (H)}{z},   
\end{equation}
and hence 
$$
\lim_{z \to0} z \RC_a(z) = \text{the multiplicity of the eigenvalue  zero} - \dim(H). 
$$

\begin{remark}
By \cite[Corollary 2.2]{CHS18}, the tracial moments $\Tr(a^n)$ for all but finitely many natural numbers $n$ determine the eigenvalues of $a$. 
So, for any $p \in \N$, we can generalize the above setting to the Schatten class
$S_p$ by using the truncated generating function 
$$
\RC_p(z) = \sum_{n=p}^\infty \frac{\omega(a^n)}{z^{n+1}}. 
$$
\end{remark}

Let $a$ and $b$ be cyclic-Boolean independent in $(\mathcal A, \varphi,\omega)$. 
It is known \cite{SW97} (and will be shown in Remark~\ref{rem:booleconv} below)
that the Green function of $a+b$ can be computed via the formula
\begin{equation}\label{eq:boole_conv}
\frac{1}{G_a(z)} +\frac{1}{G_b(z)}-z = \frac{1}{G_{a+b}(z)};
\end{equation}
i.e.,
$$
B_{a+b}(z) = B_a(z)+B_b(z), 
$$
where
\begin{equation}\label{eq:Boolean_cumulant}
B_a(z) = \frac{z}{G_a(1/z)}-1
\end{equation}
is the \emph{Boolean cumulant transform}.
The next theorem generalizes this identity to an analogous formula
for the generating function $\RC_{a+b}(z)$
which gives information on the eigenvalues of $a+b$. 

\begin{theorem}\label{thm:convolution} Let $a$ and $b$ be cyclic-Boolean
  independent elements. Then the renormalized Cauchy transform of their sum is
$$
\RC_{a+b}(z) =  \RC_a(z) + \RC_b (z) + \frac{G_a'(z)}{G_a(z)} + \frac{G_b'(z)}{G_b(z)} - \frac{G_{a+b}'(z)}{G_{a+b}(z)} + \frac{1}{z};
$$
i.e., if we define {\rm(cf.~\eqref{eq:gAt=gA+DlogzG})} $$
\HH_a(z) = \RC_a(z) +\frac{d}{dz}\log zG_a(z)
$$
then
\begin{equation}\label{eq:linearization}
\HH_{a+b}(z) = \HH_a(z) + \HH_b(z). 
\end{equation}
\end{theorem}
\begin{remark}
While $\HH$ linearizes independent sums and is useful for analyzing convolutions, we will later introduce a modification which deserves to be called the cyclic-Boolean cumulant transform; see Section \ref{sec5}.   
\end{remark}

\begin{proof}[Algebraic proof]
We expand the power $(a+b)^n$ and regroup the resulting monomials into
those ending in $a$ and those ending in $b$: 
\begin{align*}
(a+b)^n
&=
a^n
+\sum_{\mathclap{\substack{k\geq1 \\ p_0\geq0,\, p_1,q_1, \dots, p_k,q_k \geq1  \\ p_0 +  p_1 + q_1 +  \cdots + p_k +q_k =n }}} a^{p_0} b^{q_1} a^{p_1} \cdots b^{q_k} a^{p_k}
+ b^n
+  \sum_{\mathclap{\substack{k\geq1 \\  q_0 \geq0,\, p_1, q_1, \dots, p_k,q_k \geq1  \\  q_0+ q_1 + p_1+ \cdots + p_k + q_k=n }}} b^{q_0}a^{p_1} b^{q_1}\cdots a^{p_k} b^{q_k}, 
\end{align*}
and applying $\omega$ yields 
\begin{align*}
\omega((a+b)^n) 
&=
\omega(a^n) + \omega(b^n)
+\sum_{\mathclap{\substack{k\geq1 \\ p_0\geq0,\, p_1,q_1, \dots, p_k,q_k \geq1   \\ p_0 +  p_1 + q_1 +  \cdots + p_k +q_k =n }}} 
\omega(a^{{p_0}+p_k} b^{q_1} a^{p_1} \cdots b^{q_k})
+\sum_{\mathclap{\substack{k\geq1 \\ q_0 \geq0,\,  p_1, q_1, \dots, p_k,q_k \geq1  \\  q_0+ p_1 + q_1+ \cdots + p_k + q_k=n }}} 
  \omega(b^{q_0+q_k}a^{p_1} b^{q_1}\cdots a^{p_k})
\\
&=
\omega(a^n) + \omega(b^n)
+\sum_{\mathclap{\substack{k\geq1 \\ p_0\geq0,\, p_1,q_1, \dots, p_k,q_k \geq1   \\ p_0 +  p_1 + q_1 +  \cdots + p_k +q_k =n }}} 
\varphi(a^{{p_0}+p_k})\,\varphi(b^{q_1})\,\varphi(a^{p_1}) \cdots\varphi(b^{q_k}) \\
&\quad+\sum_{\mathclap{\substack{k\geq1 \\ q_0 \geq0,\,  p_1, q_1, \dots, p_k,q_k \geq1  \\  q_0+ p_1 + q_1+ \cdots + p_k + q_k=n }}} 
  \varphi(b^{q_0+q_k})\,\varphi(a^{p_1})\,\varphi(b^{q_1})\cdots \varphi(a^{p_k}). 
\end{align*}
Multiplying the above identity by $z^{-n-1}$ and taking the summation over $n$ yields 
\begin{align*}
  \RC_{a+b}(z)
  &= \sum_{n\geq0}\frac{\omega((a+b)^n)}{z^{n+1}}
  \\
  &= \sum_{n\geq1}\frac{\omega(a^n)}{z^{n+1}} +  \sum_{n\geq1}\frac{\omega(b^n)}{z^{n+1}}
+\sum_{\mathclap{\substack{k\geq1 \\ p_0\geq0,\, p_1,q_1, \dots, p_k,q_k \geq1}}}
\frac{\varphi(a^{{p_0}+p_k})}{z^{p_0+p_k+1}}\,\frac{\varphi(b^{q_1})}{z^{q_1}}\,\frac{\varphi(a^{p_1})}{z^{p_1}} \cdots\frac{\varphi(b^{q_k})}{z^{q_k}} \\
&\quad 
+\sum_{\mathclap{\substack{k\geq1 \\ q_0 \geq0,\,  p_1, q_1, \dots, p_k,q_k \geq1 }}}
  \frac{\varphi(b^{q_0+q_k})}{z^{q_0+q_k+1}}\,\frac{\varphi(a^{p_1})}{z^{p_1}}\,\frac{\varphi(b^{q_1})}{z^{q_1}}\cdots \frac{\varphi(a^{p_k})}{z^{p_k}}. 
\end{align*}
Now
\begin{align*}
  \sum_{\substack{p_0\geq0\\p\geq1}} \frac{\varphi(a^{p_0+p_1})}{z^{p_0+p+1}}
  &= \sum_{m\geq 1}   \sum_{\substack{p_0\geq0\\p\geq1\\ p_0+p=m}} \frac{\varphi(a^m)}{z^{m+1}}
 = \sum_{m\geq 1}   m \frac{\varphi(a^m)}{z^{m+1}}
  = -(zG_a(z))'
\end{align*}
and thus

\begin{align*}
  \RC_{a+b}(z)
  &= \RC_a(z) + \RC_b(z)
     - \sum_{k\geq1} (zG_a(z))'(zG_a(z)-1)^{k-1}(zG_b(z)-1)^k \\
    &\quad 
     - \sum_{k\geq1} (zG_b(z))'(zG_a(z)-1)^k(zG_b(z)-1)^{k-1}
     \\
  &= \RC_a(z) + \RC_b(z)
     - \frac{
       (zG_a(z))'(zG_b(z)-1) +(zG_a(z)-1)(zG_b(z))'
     }{ 1 - (zG_a(z)-1)(zG_b(z)-1)}
  \\
&= \RC_a(z) + \RC_b(z)
     -
     \frac{(zG_a(z))'}{G_a(z)} 
     \frac{z-F_b(z)}{F_a(z)F_b(z)-(z-F_a(z))(z-F_b(z))} \\
  &\quad   -
     \frac{(zG_b(z))'}{G_b(z)} 
     \frac{z-F_a(z)}{F_a(z)F_b(z)-(z-F_a(z))(z-F_b(z))}
     \\
&= \RC_a(z) + \RC_b (z) - \left(\frac{1}{z} + \frac{G_a'}{G_a} \right) \frac{z-F_b}{F_a + F_b -z} - \left(\frac{1}{z} + \frac{G_b'}{G_b} \right) \frac{z-F_a}{F_a + F_b -z} \\
&=  \RC_a(z) + \RC_b (z) + \left( \frac{F_a'}{F_a} (z-F_b)+ \frac{F_b'}{F_b}(z-F_a) -1\right) \frac{1}{F_a + F_b-z} + \frac{1}{z}\\
&=  \RC_a(z) + \RC_b (z) + \left( \frac{F_a'}{F_a} (z-F_a-F_b)+ \frac{F_b'}{F_b}(z-F_a-F_b) -1 +F_a' +F_b'\right) \frac{1}{F_a + F_b-z} + \frac{1}{z} \\
&=  \RC_a(z) + \RC_b (z) - \frac{F_a'}{F_a} - \frac{F_b'}{F_b} + \frac{F_{a+b}'}{F_{a+b}} + \frac{1}{z} \\
&=  \RC_a(z) + \RC_b (z) + \frac{G_a'}{G_a} + \frac{G_b'}{G_b} - \frac{G_{a+b}'}{G_{a+b}} + \frac{1}{z}.  
\end{align*}
\end{proof}

\begin{proof}[Analytic proof in the setting of Example \ref{ex:star_product}]

Under the assumption that our $\ast$-algebras are represented as trace class operators 
  on the star product Hilbert space $$
  \hilbH=\C \xi \oplus\bub{\hilbH}_1\oplus\bub{\hilbH}_2
  $$
equipped with a vacuum state $\varphi = \langle . \xi, \xi \rangle$ and the trace $\omega = \Tr$ we can use this decomposition and represent the involved operators
  as block operator matrices
 \begin{equation}\label{eq:block}
  A =
  \begin{bmatrix}
    \alpha & a' & 0\\
    a& \bub{A}&0\\
    0&0&0
  \end{bmatrix}, 
  \qquad
  B =
  \begin{bmatrix}
    \beta &0& b'\\
    0&0&0\\
    b&0& \bub{B}
  \end{bmatrix}, 
  \qquad
  A+B =
  \begin{bmatrix}
    \alpha+\beta & a' & b'\\
    a& \bub{A}&0\\
    b&0&\bub{B}
  \end{bmatrix}. 
\end{equation}
   In other words, $(A+B)\bub{}=\bub{A}\oplus\bub{B}$
  is a direct sum and therefore
  $$
  \RC_{(A+B)\bub{}}(z) = \RC_{\bub{A}}(z) + \RC_{\bub{B}}(z)
  $$
  and we conclude with the identity
  \eqref{eq:gAt=gA+DlogzG}.
\end{proof}

\begin{remark}
\label{rem:booleconv}
  The idea of the above algebraic/analytic proofs can also be used to verify the known formula \eqref{eq:boole_conv}. For example, the Banachiewicz formula \eqref{eq:Banachiewicz} applied to the decomposition \eqref{eq:block} yields the Green function in the form 
\[
\frac1{G_{A+B}(z)} =  z - (\alpha+\beta) - a' (z-\bub{A})^{-1}a - b' (z-\bub{B})^{-1} b.  
\] 
Combining this with the formulas
\[
\frac1{G_A(z)} = z - \alpha - a' (z-\bub{A})^{-1}a \qquad {\rm and} \qquad \frac1{G_B(z)} = z - \beta - b' (z-\bub{B})^{-1}b
\]
we obtain \eqref{eq:boole_conv}.   
\end{remark}

\subsection{Examples from star product graphs}

For a  rooted graph $(\bG,o)$, its $N$-fold star product
$(\bG_N,o_N)=(\bG,o)\starprod (\bG,o) \starprod \cdots \starprod (\bG,o)$ has
the adjacency matrix that is the sum of cyclic-Boolean independent copies of the adjacency matrix of $\bG$; see \eqref{eq:boole_sum}. 
Therefore, Theorem \ref{thm:convolution} and \eqref{eq:FGamma1.Gamma2} imply that  
\begin{equation}\label{eq:renomormalized_Cauchy}
\RC_{\bG_N}(z)=N \RC_a(z)+ N\frac{G_\bG'(z)}{G_\bG(z)}- \frac{G_{\bG_N}'(z)}{G_{\bG_N}(z)} + \frac{N-1}{z}
\end{equation}
where
\begin{equation}\label{eq:Green}
\frac{1}{G_{\bG_N}(z)}=\frac{N}{G_\bG(z)}-(N-1)z.
\end{equation}

\begin{example}[Star graph]\label{exa:star_graph}

The star graph $(\bS_N,o_N)$ on $N+1$ vertices $\{0,1,\dots, N\}$ has edges $\{0,i\}, i=1,2,\dots, N$. It is the $N$-fold star product of the complete graph $(\bK_2,o)$ (Figure \ref{star_graph}). 
The eigenvalues of the adjacency matrix of $\bK_2$ are $\pm1$, and hence
$$
\RC_{\bK_2}(z)=\frac{1}{z-1}+\frac{1}{z+1}-\frac2{z}\,\qquad {\rm and} \qquad G_{\bK_2}(z)=\frac{1}{2}\left(\frac{1}{z-1}+\frac{1}{z-1}\right), 
$$
where the latter formula can be computed via \eqref{eq:phiGamma=GGammaphiGamma-o}. 
Using \eqref{eq:Green} entails
$$ 
G_{\bS_N}(z)=\frac{1}{2}\left(\frac{1}{z - \sqrt{N}} - \frac{1}{z+\sqrt{N}}\right)  \,\qquad {\rm and} \qquad \frac{G_{\bS_N}'(z)}{G_{\bS_N}(z)}=\frac{1}{z}-\frac{1}{z - \sqrt{N}} - \frac{1}{z+\sqrt{N}}.   
$$
The renormalized Cauchy transform of $(\bS_N,o_N)$ may be calculated from \eqref{eq:renomormalized_Cauchy} as follows: 
\begin{align*}
\RC_{\bS_N}(z)&=N\left(\frac{1}{z - 1}+\frac{1}{z+1}-\frac{2}{z}\right)+N\left(\frac{1}{z}-\frac{1}{z - 1} - \frac{1}{z+1} \right) \\
 &\quad+\left(\frac{1}{z - \sqrt{N}} + \frac{1}{z+\sqrt{N}} -\frac{1}{z}\right)+ \frac{N-1}{z}\\
 & =\frac{1}{z - \sqrt{N}} +\frac{1}{z+\sqrt{N}} -\frac{2}{z}. 
\end{align*}
The Cauchy transform is given by 
$$\Cauchy_{\bS_N}(z)=\frac{1}{z - \sqrt{N}} +\frac{1}{z+\sqrt{N}}+\frac{N-1}{z}.$$
This recovers the fact that the multiset of eigenvalues of the adjacency matrix of $\bS_N$ is given by  $\{[\sqrt{N}]^1,[- \sqrt{N}]^1, [0]^{N-1}\}$.
\end{example}

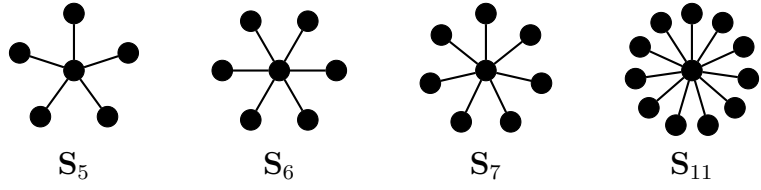
\begin{figure}
\begin{tikzpicture}
[mystyle/.style={scale=0.7, draw,shape=circle,fill=black}]
\def\ngon{5}
\def\ngonn{6}
\node[regular polygon,regular polygon sides=\ngon,minimum size=1.5cm] (p) {};
\foreach\x in {1,...,\ngon}{\node[mystyle] (p\x) at (p.corner \x){};}
\node[mystyle] (p0) at (0,0) {};
\foreach\x in {1,...,\ngon}
{
 \draw[thick] (p0) -- (p\x);
}
  \node [label=below:$\bS_{\ngon}$] (*) at (0,-0.8) {};
 \end{tikzpicture}
  \qquad
\begin{tikzpicture}
[mystyle/.style={scale=0.7, draw,shape=circle,fill=black}]
\def\ngon{6}
\def\ngonn{7}
\node[regular polygon,regular polygon sides=\ngon,minimum size=1.5cm] (p) {};
\foreach\x in {1,...,\ngon}{\node[mystyle] (p\x) at (p.corner \x){};}
\node[mystyle] (p0) at (0,0) {};
\foreach\x in {1,...,\ngon}
{
 \draw[thick] (p0) -- (p\x);
}
  \node [label=below:$\bS_{\ngon}$] (*) at (0,-0.8) {};
 \end{tikzpicture}
  \qquad
\begin{tikzpicture}
[mystyle/.style={scale=0.7, draw,shape=circle,fill=black}]
\def\ngon{7}
\def\ngonn{8}
\node[regular polygon,regular polygon sides=\ngon,minimum size=1.5cm] (p) {};
\foreach\x in {1,...,\ngon}{\node[mystyle] (p\x) at (p.corner \x){};}
\node[mystyle] (p0) at (0,0) {};
\foreach\x in {1,...,\ngon}
{
 \draw[thick] (p0) -- (p\x);
}
  \node [label=below:$\bS_{\ngon}$] (*) at (0,-0.8) {};
 \end{tikzpicture}
  \qquad
  \begin{tikzpicture}
[mystyle/.style={scale=0.7, draw,shape=circle,fill=black}]
\def\ngon{11}
\def\ngonn{12}
\node[regular polygon,regular polygon sides=\ngon,minimum size=1.5cm] (p) {};
\foreach\x in {1,...,\ngon}{\node[mystyle] (p\x) at (p.corner \x){};}
\node[mystyle] (p0) at (0,0) {};
\foreach\x in {1,...,\ngon}
{
 \draw[thick] (p0) -- (p\x);
}
  \node [label=below:$\bS_{\ngon}$] (*) at (0,-0.8) {};
 \end{tikzpicture}
  \caption{Star graphs}\label{star_graph}

\end{figure}

\begin{example}[Friendship graph]\label{exa:friendship_graph}

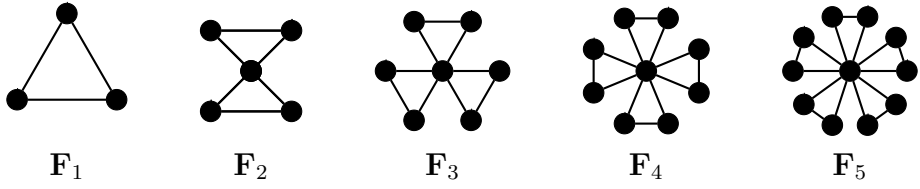
\begin{figure}
\begin{tikzpicture}
[mystyle/.style={scale=0.7, draw,shape=circle,fill=black}]
\def\ngon{3}
\node[regular polygon,regular polygon sides=\ngon,minimum size=1.5cm] (p) {};
\node[regular polygon,regular polygon sides=2*\ngon,minimum size=1.5cm] (q) {};
\foreach\x in {1,...,\ngon}{\node[mystyle] (p\x) at (p.corner \x){};}
\foreach\x in {1,...,\numexpr\ngon-1\relax}{
  \foreach\y in {\x,...,\numexpr\x+1\relax}{
    \draw[thick] (p\x) -- (p\y);
  }
}
 \draw[thick] (p1) -- (p\ngon);
 \node [label=below:$\bF_{1}$] (*) at (0,-0.8) {};
\end{tikzpicture}
\qquad
\begin{tikzpicture}
[mystyle/.style={scale=0.7, draw,shape=circle,fill=black}]
\def\ngon{4}
\node[regular polygon,regular polygon sides=\ngon,minimum size=1.5cm] (p) {};
\node[regular polygon,regular polygon sides=\ngon,minimum size=.75cm] () {};
\foreach\x in {1,...,\ngon}{\node[mystyle] (p\x) at (p.corner \x){};}
\foreach\x in {1,3}{
  \foreach\y in {\x,...,\numexpr\x+1\relax}{
    \draw[thick] (p\x) -- (p\y);
  }
}
\node[mystyle] (p0) at (0,0) {};
\foreach\x in {1,...,\ngon}
{
 \draw[thick] (p0) -- (p\x);
}

\foreach\x in {1,3}{
  \foreach\y in {\x,...,\numexpr\x+1\relax}{
    \draw[thick] (p\x) -- (p\y);
  }
}
\node[mystyle] (p0) at (0,0) {};
\foreach\x in {1,...,\ngon}
{
 \draw[thick] (p0) -- (p\x);
}

\node[mystyle] (q0) at (0,0) {};

 \node [label=below:$\bF_{2}$] (*) at (0,-0.8) {};
 \end{tikzpicture}
 \qquad
\begin{tikzpicture}
[mystyle/.style={scale=0.7, draw,shape=circle,fill=black}]
\def\ngon{6}
\node[regular polygon,regular polygon sides=\ngon,minimum size=1.5cm] (p) {};
\foreach\x in {1,...,\ngon}{\node[mystyle] (p\x) at (p.corner \x){};}
\foreach\x in {1,3,5}{
  \foreach\y in {\x,...,\numexpr\x+1\relax}{
    \draw[thick] (p\x) -- (p\y);
  }
}
\node[mystyle] (p0) at (0,0) {};
\foreach\x in {1,...,\ngon}
{
 \draw[thick] (p0) -- (p\x);
}
 \node [label=below:$\bF_{3}$] (*) at (0,-0.8) {};
  \end{tikzpicture}
 \qquad
 \begin{tikzpicture}
[mystyle/.style={scale=0.7, draw,shape=circle,fill=black}]
\def\ngon{8}
\node[regular polygon,regular polygon sides=\ngon,minimum size=1.5cm] (p) {};
\foreach\x in {1,...,\ngon}{\node[mystyle] (p\x) at (p.corner \x){};}
\foreach\x in {1,3,5,7}{
  \foreach\y in {\x,...,\numexpr\x+1\relax}{
    \draw[thick] (p\x) -- (p\y);
  }
}
\node[mystyle] (p0) at (0,0) {};
\foreach\x in {1,...,\ngon}
{
 \draw[thick] (p0) -- (p\x);
}
 \node [label=below:$\bF_{4}$] (*) at (0,-0.8) {}; \end{tikzpicture}
  \qquad
 \begin{tikzpicture}
[mystyle/.style={scale=0.7, draw,shape=circle,fill=black}]
\def\ngon{10}
\node[regular polygon,regular polygon sides=\ngon,minimum size=1.5cm] (p) {};
\foreach\x in {1,...,\ngon}{\node[mystyle] (p\x) at (p.corner \x){};}
\foreach\x in {1,3,5,7,9}{
  \foreach\y in {\x,...,\numexpr\x+1\relax}{
    \draw[thick] (p\x) -- (p\y);
  }
}
\node[mystyle] (p0) at (0,0) {};
\foreach\x in {1,...,\ngon}
{
 \draw[thick] (p0) -- (p\x);
}
 \node [label=below:$\bF_{5}$] (*) at (0,-0.8) {};
 \end{tikzpicture}
 \caption{Friendship graphs} \label{Friendship}
\end{figure}

The friendship graph $\bF_N$ is the graph  with $2N+1$ vertices $\{0, \dots, 2N \}$ in which $0$ is connected to every other vertex and the only other edges are $\{2i-1,2i\}$ for $1\leq i \leq N$. The friendship graph is the $N$-fold star product of the complete graph $(\bK_3,o)$ with itself; see Figure \ref{Friendship}. 

In this case 
$$\RC_{\bK_3}(z)=\frac{2}{z+1}+\frac{1}{z-2}-\frac3{z}\,\qquad {\rm and} \qquad G_{\bK_3}(z)=\frac1{3}\left(\frac{2}{z+1}+\frac{1}{z-2}\right), $$
from which
$$ \frac{G_{\bK_3}'(z)}{G_{\bK_3}(z)}=\frac{-z^2 + 2 z - 3}{z^3 - 2 z^2 - z + 2}=\frac1{z - 1} - \frac1{z + 1} - \frac1{z - 2}.$$
On the other hand $$\frac{1}{G_{\bF_N}(z)}=\frac{N}{G_{\bK_3}(z)}-(N-1)z=\frac{z(z-1) -2 N }{z - 1}, $$
and then
$$\frac{G_{\bF_N}'(z)}{G_{\bF_N}(z)}=\frac{1}{z-1} + \frac{1 - 2 z}{z^2 - z-2N} =\frac{1}{z-1} -\frac{1}{z-(1+\sqrt{1+8N})/2} -\frac{1}{z-(1-\sqrt{1+8N})/2}. $$
Thus the renormalized Cauchy transform may be calculated as follows.
\begin{align*}
&\RC_{\bF_N}(z) \\ 
&=N\left(\frac{2}{z+1}+\frac{1}{z-2}-\frac{3}{z}\right)+N\left(\frac{1}{z - 1} - \frac{1}{z + 1} + \frac{1}{2 - z} \right) \\ 
&\quad -\left(\frac{1}{z-1} -\frac{1}{z-(1+\sqrt{1+8N})/2} -\frac{1}{z-(1-\sqrt{1+8N})/2}\right)+ \frac{N-1}{z}\\
&= \frac{N - 1}{z - 1} + \frac{N}{z + 1}+\frac{1}{z-(1+\sqrt{1+8N})/2} +\frac{1}{z-(1-\sqrt{1+8N})/2} -\frac{2 N +1}{z}. 
\end{align*}
Then
\begin{align*}
\Cauchy_{\bF_N}(z)&=\RC_{\bF_N}(z)+\frac{2 N +1}{z}\\
&= \frac{N - 1}{z - 1} + \frac{N}{z + 1}+\frac{1}{z-(1+\sqrt{1+8N})/2} +\frac{1}{z-(1-\sqrt{1+8N})/2}. 
\end{align*}
This recovers the fact that the multiset of eigenvalues of the adjacency matrix of $\bF_N$ is given by 
$$
\left\{\left[\frac{1}{2}-\frac{1}{2}\sqrt{1+8N}\right]^1,[-1]^N,[1]^{N-1},\left[\frac{1}{2}+\frac{1}{2}\sqrt{1+8N}\right]^1\right\}. 
$$

\end{example}

\subsection{Cyclic-Boolean central limit theorem}
Since we have an appropriate linearization \eqref{eq:linearization} for
cyclic-Boolean convolution, we are able to determine the central limit law.  
 
\begin{theorem}\label{thm:CBCLT}
 For each $N \in \N$, let $\{a_i^{(N)}\}_{i=1}^N$ be self-adjoint cyclic-Boolean independent random variables in a cncps $(\alg{A}_N,\varphi_N,\omega_N)$. Assume that, for each fixed $k \in \N$,  the moments $\varphi_N((a_i^{(N)})^k)$ and $\omega_N((a_i^{(N)})^k)$ do not depend on $i$ or $N$, and also $\omega_N(a_i^{(N)})=\varphi_N(a_i^{(N)})=0$, $\varphi_N((a_i^{(N)})^2)=1$ for all $i$ and $N$.  Then, for the normalized sum 
\[
s_N= \frac{a_1^{(N)}+a_2^{(N)}+\cdots +a_N^{(N)}}{\sqrt{N}}, 
\]
it holds that 
\[
\lim_{N\to\infty}\varphi_N(s_N^k) = \begin{cases} 1,& k \in 2\N, \\ 0, &k \in 2\N -1, \end{cases}
\quad\text{and}\quad 
\lim_{N\to\infty}\omega_N(s_N^k) = \begin{cases} 2,& k \in 2\N+2, \\ 0, &k \in 2\N +1.  \end{cases}
\]
\end{theorem}
\begin{proof}
The Boolean central limit theorem \cite{SW97} asserts that 
$G_{s_N}(z) \to z/(z^2-1)$. Let $\alpha=\omega_N((a_i^{(N)})^2)$, then Theorem \ref{thm:convolution} yields 
 $$
 \HH_{s_N}(z) = N^{3/2} \HH_{a_1^{(N)}}(\sqrt{N}z) = N^{3/2}\left( \frac{\alpha-2}{(\sqrt{N}z)^3} +O(N^{-2})  \right) \to \frac{\alpha-2}{z^3}. 
 $$
Therefore 
$$
\RC_{s_N}(z) \to \frac{2z}{z^2-1} -\frac{2}{z}+  \frac{\alpha-2}{z^3} = \sum_{n\geq1} \frac{\RC_n}{z^{n+1}},  
$$
where 
$$
\RC_n= 
\begin{cases} 0, & \text{$n$ is odd}, \\
2, & \text{$n$ is even and $n \geq4$}, \\
\alpha, & n=2.  
\end{cases}
$$
\end{proof}

The limit law exhibits a large spectral gap:
\begin{corollary}\label{cor:CBCLT} In addition to the setting of Theorem \ref{thm:CBCLT}, suppose that $\alg{A}_N=T(H_N)$ for some Hilbert space $H_N$ 
and $\omega_N= \Tr_{H_N}$. Let $\lambda_N$ and $\mu_N$ be the largest and smallest eigenvalues of $s_N$, respectively. The following assertions hold: 
\begin{enumerate}[\rm(i)]
\item the multiplicities of $\lambda_N$ and $\mu_N$ are both one for sufficiently large $N$; 
\item $\lambda_N$ converges to $1$ and $\mu_N$ converges to $-1$ as $N\to\infty$; 
\item the remaining eigenvalues accumulate around 0:
 $$
 \lim_{N\to\infty}\dist(\spec(s_N)   \setminus\{\lambda_N,\mu_N\}, 0)=0 
 $$
\end{enumerate}
\end{corollary}
\begin{proof}
Let $s$ be a self-adjoint operator of rank two on a Hilbert space $K$ having
eigenvalues $-1, 1,0$. Then $\Tr_K(s^n) = 2$ for even $n \geq2$ and $\Tr_K(s^n)
= 0$ for odd $n \geq1$. The convergence of $\Tr_H(s_N^k)$ in Theorem
\ref{thm:CBCLT} and \cite[Proposition 2.8]{CHS18} imply that $s_N \to s$ in
eigenvalues and this concludes the argument.
\end{proof}
\begin{remark} As shown in the proof of Theorem \ref{thm:CBCLT}, $\Tr_{H_N}(s_N^2)$ converges (actually is equal) to $\alpha$ which might not equal $2 = \Tr_K(s^2)$. 
This difference of Hilbert-Schmidt norms is due to a large number of small eigenvalues of $s_N$ and does not contradict the convergence of eigenvalues; see \cite[Proposition 2.8, Proposition 2.10 and Remark 2.11]{CHS18}. 
\end{remark}

Now we come back to the original model, the adjacency matrix of the star product of rooted graphs. 

\begin{corollary}\label{cor:star-CLT} Suppose that $(\bG,o)$ is a  rooted graph with $\deg(o)\ge1$. Let $A_N$ be the adjacency matrix of the $N$-fold star product graph $(\bG,o) \starprod (\bG,o) \starprod \ldots \starprod (\bG,o)$. Let $\lambda_N$ and $\mu_N$ be the largest and smallest eigenvalues of $(\deg(o)N)^{-\frac{1}{2}}A_N$, respectively.  The following assertions hold: 
\begin{enumerate}[\rm(i)]
\item the multiplicities of $\lambda_N$ and $\mu_N$ are both one for sufficiently large $N$;
\item $\lambda_N$ converges to $1$ and $\mu_N$ converges to $-1$ as $N\to\infty$; 
\item
 $\displaystyle
 \lim_{N\to\infty}\dist(\spec((\deg(o)N)^{-\frac{1}{2}}A_N)   \setminus\{\lambda_N,\mu_N\}, 0)=0 
$
\end{enumerate}
\end{corollary}
\begin{proof} This is a combination of Corollary \ref{cor:CBCLT},  formula \eqref{eq:boole_sum} and Example \ref{exa:CB}. The factor $(\deg(o)N)^{-\frac{1}{2}}$ appears because of the variance $\langle A_\bG^2\delta_o, \delta_o\rangle_{\ell^2(V)}= \deg(o)$, where $V$ is the vertex set of $\bG$.   
\end{proof}

\begin{remark} In the setting of Corollary \ref{cor:star-CLT} it is already known that, according to the Boolean central limit theorem,  the distribution of $(\deg(o)N)^{-\frac{1}{2}}A_N$ regarding the vector state $\varphi_N= \langle \cdot \delta_o, \delta_o\rangle$ converges weakly to $\frac{1}{2}(\delta_{-1} +\delta_1)$. This fact entails an intuitive consequence of Corollary \ref{cor:star-CLT}: the vector $\delta_o$ in the tensor product Hilbert space $\ell^2(V)^{\otimes N}$ is almost orthogonal to the subspace spanned by eigenvectors corresponding to small eigenvalues, or equivalently, $\delta_o$ is almost contained in the two-dimensional subspace spanned by the eigenvectors corresponding to the eigenvalues near $\pm1$.  
\end{remark}

\begin{example} Corollary \ref{cor:star-CLT} can be directly confirmed in the following examples.  
\begin{enumerate}[\rm(i)] 
\item For the star graph on $N+1$ vertices, its adjacency matrix divided by $\sqrt{N}$ has eigenvalues $\{[1]^1,[-1]^1,[0]^{N-1}\}$; see Example \ref{exa:star_graph}. Eigenvectors corresponding to the eigenvalues $1$ and $-1$ are $f_1=(\sqrt{N},1,1,\dots,1)$ and $f_2=(-\sqrt{N},1,1,\dots,1)$, respectively, and hence, the function $\delta_o$, which corresponds to the vector $(1,0,0,\dots,0)$, is exactly contained in the subspace spanned by $f_1$ and $f_2$. 

\item For the friendship graph on $2N+1$ vertices, its adjacency matrix divided by $\sqrt{2N}$ has eigenvalues 
$$
\left\{\left[-\sqrt{1+\frac{1}{8N}}-\frac{1}{2\sqrt{2N}}\right]^1, \left[-\frac{1}{\sqrt{2N}}\right]^{N}, \left[\frac{1}{\sqrt{2N}}\right]^{N-1}, \left[\sqrt{1+\frac{1}{8N}}+\frac{1}{2\sqrt{2N}}\right]^1\right\}; 
$$
see Example \ref{exa:friendship_graph}. 
\end{enumerate}

\end{example}

\section{Cyclic-Boolean cumulants}\label{sec5}

\subsection{Univariate cumulants}
Let $(\alg{A},\varphi,\omega)$ be a cncps and $a \in \alg{A}$. The generating function $\HH_a$ defined in \eqref{eq:linearization} has the series expansion 
$$
\HH_a(z) = \sum_{n=1}^\infty\frac{h_n(a)}{z^{n+1}}, 
$$
where the first two coefficients are $h_1(a) = \omega(a) -\varphi(a)$ and $h_2(a)=\omega(a^2) +\varphi(a)^2-2\varphi(a^2)$.  In general, $h_n(a)$ is of the form $\omega(a^n)-n \varphi(a^n) + (\text{polynomial on $\varphi(a), \dots, \varphi(a^{n-1})$})$. 

We can modify $\HH_a(z)$ by adding the Boolean cumulants to delete $-n \varphi(a^n)$ from $h_n(a)$. We switch from $\RC_a$ and $G_a$ to the moment generating functions 
$$
\MG_a(z) = \frac{1}{z}\RC_a\left(\frac{1}{z}\right) = \sum_{n\geq1} \omega(a^n)z^n, \qquad M_a(z) = \frac{1}{z}\left(G_a\left(\frac{1}{z}\right) -z\right) = \sum_{n\geq1} \varphi(a^n)z^n. 
$$
The Boolean cumulant transform \eqref{eq:Boolean_cumulant} is then expressed by 
$$
B_a(z) = \frac{M_a(z)}{1+M_a(z)} = \sum_{n\geq1} b_n(a)z^n. 
$$
We introduce the new generating function 
\begin{align}
\CB_a(z) &= \frac{1}{z}\HH_a\left(\frac{1}{z}\right) + zB_a'(z) = \MG_a(z) - \frac{z M_a(z)M_a'(z)}{(1+M_a(z))^2} \notag \\
&= \MG_a(z) -z M_a(z) B_a'(z), \label{eq:moment-cumulant-generating-function}
\end{align}
which linearizes the convolution 
\[\CB_{a+b}(z) = \CB_a(z)+\CB_b(z).
\]
The function $\CB_a$ will be called the \emph{cyclic-Boolean cumulant transform of $a$} and the coefficients $c_n(a)$ appearing as 
\begin{equation}\label{eq:CB_cumulants}
\CB_a(z) = \sum_{n\geq1} c_n(a)z^n 
\end{equation}
are called the (univariate) \emph{cyclic-boolean cumulants} of $a$.  The first two cumulants are 
\[
c_1(a)=\omega(a) \qquad{\rm and} \qquad c_2(a) = \omega(a^2) - \varphi(a)^2.
\]
 For general $n \geq2$, there exists a universal polynomial $P_n(x_1,\dots, x_{n-1})$ depending only on $n$ such that 
 \[
 c_n(a) =\omega(a^n) + P_n(\varphi(a), \dots, \varphi(a^{n-1})).
 \]

\subsection{Cyclic-interval partitions}
Cyclic-Boolean independence gives rise to an exchangeability system and 
we can define and compute the (multivariate) cyclic-Boolean cumulants using the methods of
\cite{Lehner:2004:cumulants1,HasebeLehner:cumulants5}.  
The relevant partition structure turns out to be cyclic-interval partitions, which were already discussed
in \cite{DAGSV} in their search for notions of independence, similar to Boolean and monotone ones, but such that the algebra of scalars, $\mathbb{C}$, is independent from any other algebra.

Before embarking on we recall some basic concepts on set partitions. 

\begin{definition} Let $k \in \N$. We often use the notation $[k]=\{1,2,\dots,k\}$. 
\begin{enumerate}[\rm(i)]
\item A \emph{set partition} of $[k]$ is a set $\pi=\{B_1,B_2,\dots, B_p\}$ of
 nonempty and  disjoint subsets $B_1, \dots, B_p$ of $[k]$, called
 \emph{blocks}, such that  their
 union is $[k]$. The length  $|\pi|$   of a partition $\pi$ is
 the number of blocks.  The set of the partitions of the set $[k]$ is denoted
 by $\SP(k)$. Set partitions are in one-to-one
 correspondence with equivalence relations:
 Any set partition 
 $\pi\in \SP(k)$ determines an equivalence relation $i \sim_\pi j$ on $[k]$ by
 requiring that $i, j$ belong to the same block of $\pi$; conversely, for an equivalence relation $\sim$ on $[k]$ its equivalence classes determine disjoint subsets of $[k]$ and hence a set partition. 

\item A subset of $[k]$ of form $\{i,i+1,\dots, j\}$ is called an \emph{interval}
 and a set partition of $[k]$ is called an \emph{interval partition} if all its blocks are intervals. The set of the interval partitions is denoted by $\Int(k).$

\item For set partitions $\sigma, \pi \in \SP(k)$ we write $\sigma \le \pi$ if every block of $\sigma$ is a subset of a block of $\pi$. This makes $\SP(k)$ a poset. The trivial set partition $\{[k]\}$ is the maximum of $\SP(k)$, which is denoted by $\hat1_k$.

\item A tuple $(i_1, \dots, i_k) \in \N^k$ induces a unique equivalence
 relation $ \sim $ on $[k]$ by the requirement that $p\sim q$ holds if and only
 if $i_p=i_q$. The corresponding set partition is called the
 \emph{kernel set partition}, denoted by $\kappa(i_1, \dots, i_k)$.

\end{enumerate}
\end{definition}
\begin{example} Some kernel set partitions are
\begin{align*}\label{eq:sp}
\kappa(6,3,2,3,6) &= \{\{3\},\{2,4\},\{1,5\}\}, \\
\kappa(2,7,4,7,4,2,4)&=\{\{1,6\},\{3,5,7\},\{2,4\}\}. 
\end{align*}
\end{example}

We will see that cyclic-Boolean cumulants $\CBC_\pi$ (defined in the next section) vanish identically unless $\pi$ is a \emph{cyclic-interval partition}.
\begin{definition}
  A partition $\pi\in\SP(n)$ is called a \emph{cyclic-interval partition}
  if every block is an interval or the complement of an interval.
  In other words, there is a cyclic permutation $\sigma\in\SG_n$ such that
  $\sigma\cdot\pi$ is an interval partition.
We denote by $\CI(n)$ the set of the cyclic-interval partitions of $[n]$.  
\end{definition}

As already noticed in \cite[Corollary 1]{DAGSV}, it is not difficult to see   
that the number of cyclic-interval partitions is $\abs{\CI(n)}=2^n-n$. 
To see this, the most convenient picture of cyclic-interval partitions is obtained
  by actually drawing them on a circle as show in Fig.~\ref{fig:cyclicintpart}.
  Then it is clear that a cyclic-interval partition is uniquely determined
  by the set of separators of the blocks. For the maximal partition $\hat{1}_n$
  this set is empty, while for all other cyclic-interval partitions there
  must be at least two separators. 
    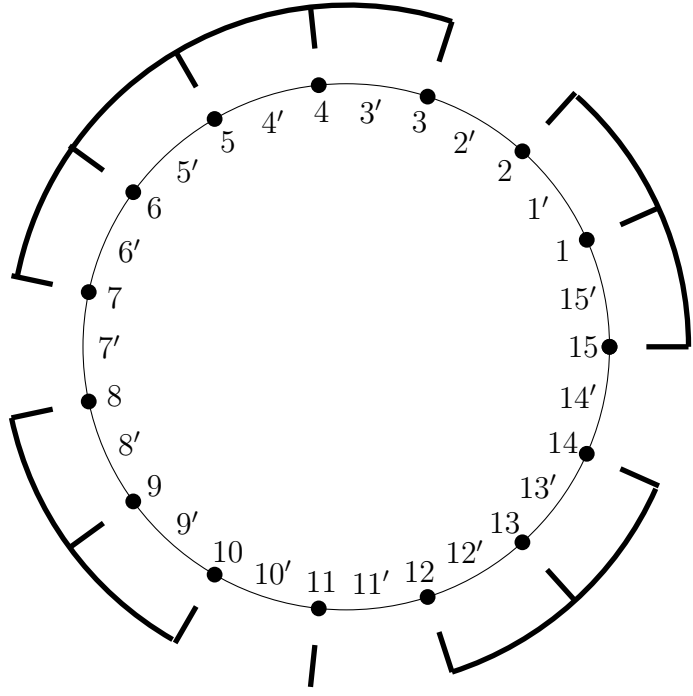
\begin{figure}
    \centering
  \begin{tikzpicture}[scale=0.7]
 \coordinate (C) at (0,0);
\def\N{15};
\def\rad{5};

 \coordinate (C0) at (0*360/\N:6.5);
\draw[line width=2] (C0) arc(0*360/\N : 2*360/\N : 6.4);

 \coordinate (C1) at (3*360/\N:6.5);
\draw[line width=2] (C1) arc(3*360/\N : 7.05*360/\N : 6.4);

 \coordinate (C2) at (8*360/\N:6.5);
\draw[line width=2] (C2) arc(8*360/\N : 10*360/\N : 6.4);

 \coordinate (C4) at (12*360/\N:6.5);
\draw[line width=2] (C4) arc(12*360/\N : 14*360/\N : 6.4);

\draw (0,0) circle (\rad);

\foreach \i in {0,...,\N}
\fill  (360/\N*\i:5) circle (0.15);

\foreach \i in {1,...,\N}
 \coordinate[label=center:$\i$] (D) at (360/\N*\i:4.5);

\foreach \i in {1,...,\N}
 \draw[line width=2]  (360/\N*\i:5.7) -- (360/\N*\i:6.5);

\foreach \i in {1,...,\N}
\coordinate[label=center:$\i'$] (D) at (360/\N*0.5+360/\N*\i:4.5);

\end{tikzpicture}
    \caption{The cyclic interval partition $\{1,2,15/3,4,5,6,7/8/9/10/11/12/13/14\}$ has separating set $\{2',7',10',11',14'\}$}
    \label{fig:cyclicintpart}
  \end{figure}

\subsection{Multivariate cumulants }

In order to avoid the discussion of positivity (see Remark~\ref{rem:positivity}
below) we notice  that one can
easily extend the definition of independence to a purely algebraic  setting without
positivity. Thus in this section we will
focus on an algebraic cyclic probability space
$(\alg{A},\varphi,\omega)$ without positivity structure, that is, $\alg{A}$ is
an algebra over $\C$, $\varphi$ is a linear functional and $\omega$ is a
tracial linear functional.

Take copies $\alg{A}_k$ of $\alg{A}$ and define the nonunital algebraic free product 
$$
\alg{U}=\mathop{\ast}_{k \in \N} \alg{A}_k = \bigoplus_{n\in \N} \bigoplus_{\substack{(k_1,\dots, k_n) \in \N^n \\ k_1 \neq \cdots \neq k_n}} \alg{A}_{k_1}\otimes  \cdots \otimes \alg{A}_{k_n} . 
$$ 
Let $\pi^{(i)}\colon a \mapsto a^{(i)}$ denote the embedding of $\alg{A}$ into $\alg{U}$ as the
$i$-th copy $\alg{A}_i$. By the universality of tensor products we can define
$(\tilde \varphi, \tilde \omega)$ on $\alg{U}$ as follows: for $n\geq1$, an alternating tuple $(k_1,\dots, k_n) \in \N^n$ and $a_i \in \mathcal{A}_{k_i}, i=1,2,\dots,n$, set 
\begin{align*}
\tilde \varphi(a_1 a_2\cdots a_n) &:= \varphi(a_1)\varphi(a_2) \cdots \varphi(a_n),   \\
\tilde \omega(a_1 a_2 \cdots a_n) &:=  
\begin{cases} 
\omega(a_1), &n=1, \\ 
\varphi(a_1)\varphi(a_2) \cdots \varphi(a_n), & n \geq2, k_1 \neq k_n, \\ 
 \varphi(a_n a_1) \varphi(a_2) \cdots \varphi(a_{n-1}), & n \geq2, k_1 = k_n. 
\end{cases}
\end{align*}
One can check that $\tilde \omega$ is a trace on $\alg{U}$, $\varphi = \tilde
\varphi \circ \pi^{(i)}$, $\omega = \tilde \omega \circ \pi^{(i)}$ and the
family of subalgebras $\{\pi^{(i)}(\alg{A})\}_{i=1}^\infty$ is cyclic-Boolean
independent in $(\alg{U}, \tilde\varphi,\tilde\omega)$.

\begin{remark}[Positivity] \label{rem:positivity}
  It is not clear under what conditions the product trace
  $\tilde{\omega}$ preserves positivity.
  First observe that the trivial example $\omega=0$ shows
  that some conditions are necessary.
  Indeed,
  if $\omega=0$ and $a,b\in\alg{A}$ are self-adjoint then
  $\tilde\omega((a^{(1)}+ b^{(2)})^2)= 2\varphi(a)\varphi(b)$,
  which can be   negative, although both $\varphi$ and $\omega$
  are positive.

  This example suggests that in order to expect positivity
  of $\tilde{\omega}$ one should at least require
  $\omega(x^*x)\geq \abs{\varphi(x)}^2$
  or
  $\omega(x^*x)\geq \varphi(x^*x)$.
  Both conditions however are not promoted
  to the cyclic free product.
  Although  the proof of
  \cite[Theorem 2.2]{BozejkoLeinertSpeicher:1996:convolution}
  adapts well to show
  $\tilde{\omega}(x^*x)\geq \abs{\tilde{\varphi}(x)}^2$ 
  for $x\in \alg{U}_{11}$ or $x\in \alg{U}_{22}$,
  where
  $$\alg{U}_{11} = \bigoplus_{n\geq 3} 
  \bigoplus_{\substack{(k_1,\dots, k_n) \in \N^n \\ k_1 \neq \cdots \neq k_n\\ k_1=k_n=1}} \alg{A}_{k_1}\otimes  \cdots \otimes \alg{A}_{k_n}
  $$
  etc.,
  the following example shows that positivity fails on elements mixing
  these subspaces.
  Choose $x,y,z\in \alg{A}_1$ and $w\in\alg{U}_{22}$
  and put $a = x+ywz$. Note that $ywz$ is alternating and we compute
  \begin{align*}
  \tilde{\omega}(a^*a)
  &= \omega(x^*x)
    + \tilde{\omega}(x^*ywz)
    + \tilde{\omega}(z^*w^*y^*x)
    + \tilde{\omega}(z^*w^*y^*ywz)\\
  &= \omega(x^*x) + \varphi(zx^*y)\,\varphi(w)+\varphi(w^*)\,\varphi(y^*xz^*)+\varphi(zz^*)\,\varphi(w^*)\,\varphi(y^*y)\,\varphi(w)
  \end{align*}
  and
  \begin{align*}
  \tilde{\varphi}(a^*a)
  &= \varphi(x^*x)
    + \tilde{\varphi}(x^*ywz)
    + \tilde{\varphi}(z^*w^*y^*x)
    + \tilde{\varphi}(z^*w^*y^*ywz)\\
  &= \varphi(x^*x) + \varphi(x^*y)\,\varphi(w)\,\varphi(z)+\varphi(z^*)\,\varphi(w^*)\,\varphi(y^*x)+\varphi(z^*)\,\varphi(w^*)\,\varphi(y^*y)\,\varphi(w)\,\varphi(z)
  \end{align*}
  Now choose $x,y,z,w$ such that
  $$
  \omega(x^*x)=\varphi(x^*x), 
  \qquad
  \varphi(z)=0, 
  \qquad
  \varphi(zx^*y),\varphi(w)\in\IR\setminus\{0\}. 
  $$
  Such a choice is possible, e.g., when $\omega=\varphi$ is a tracial state:
  first choose a unitary $u$ such that $\varphi(u)=0$
  and set $z=u$, $y=u^*$,
  then for selfadjoint $w$ and $x$
  we obtain
  $$
  \tilde{\omega}(a^*a)-  \tilde{\varphi}(a^*a)
  = 2\varphi(x)\varphi(w)+\abs{\varphi(w)}^2
  $$
  which can be made negative by an appropriate rescaling of  $x$. 
  
 The positivity condition $\omega(x^*x)\geq \abs{\varphi(x)}^2$ proves to be inappropriate as well. In the last specification, we further choose $x$ so that $|\varphi(x)|^2< \varphi(x^2) < 2 |\varphi(x)|^2$. Then 
  \[
 \tilde{\omega}(a^*a)-  |\tilde{\varphi}(a)|^2 
  = \varphi(x^2)- |\varphi(x)|^2 +  2\varphi(x)\varphi(w)+\abs{\varphi(w)}^2. 
  \]
  Replacing $x$ with $\lambda x, \lambda \in \R$, will change this value into 
  \[
  \left[ \varphi(x^2)- |\varphi(x)|^2\right] \left[ \lambda +  \frac{\varphi(x)\varphi(w)}{\varphi(x^2)- |\varphi(x)|^2}\right]^2 + \frac{|\varphi(w)|^2[\varphi(x^2)-2|\varphi(x)|^2]}{\varphi(x^2)- |\varphi(x)|^2}. 
  \] 
  Taking $\lambda = - \frac{\varphi(x)\varphi(w)}{\varphi(x^2)- |\varphi(x)|^2}$ will make this value negative.  
\end{remark}

Both pairs $(\alg{U},\tilde \varphi)$ and $(\alg{U},\tilde \omega)$ are
exchangeability systems in the sense of \cite[Definition
1.8]{Lehner:2004:cumulants1} except that we are not assuming unitality, which
however is not essential for the theory of cumulants. For the first pair we will get Boolean cumulants $B_\pi$ which are well known; therefore we will focus on $(\alg{U},\tilde \omega)$ from now on. 
The exchangeability of $(\alg{U},\tilde \omega)$ means that, for any $n\in \N$ and $a_1,\dots, a_n \in \alg{A}$, the value of the function 
\[
\N^n \ni (i_1,\dots, i_n)\mapsto\tilde \omega(a_1^{(i_1)} a_2^{(i_2)} \cdots a_n^{(i_n)}) \in \C
\]
is determined by the kernel set partition $\pi=\kappa(i_1,\dots, i_n)$. This value is denoted by $\omega_\pi(a_1,\dots, a_n)$, which then gives an $n$-linear functional $\omega_\pi\colon \alg{A}^n\to \C$. For each $\pi \in \SP(n)$ a partitioned cumulant $\CBC_\pi\colon \alg{A}^n \to \C$ is then defined by
\begin{align}\label{eq:c-m}
\CBC_\pi(a_1,a_2,\dots,a_n) 
= \sum_{\substack{\rho \in \SP(n) \\ \rho\leq\pi}} \omega_\rho(a_1,a_2,\dots,a_n)\,\mu(\rho,\pi), 
\end{align}
where $\mu$ is the M\"{o}bius function for the poset $\SP(n)$. By \cite[Lemma~4.18~(ii)]{HasebeLehner:cumulants5} or imitating the proof of \cite[Proposition 4.11]{Lehner:2004:cumulants1} we can prove that $\CBC_\pi =0$ if $\pi \in \SP(n) \setminus \CI(n)$. This vanishing property and the M\"{o}bius inversion of \eqref{eq:c-m} imply that 
\begin{equation}\label{eq:m-c}
\omega(a_1 a_2 \cdots a_n) = \sum_{\pi \in \CI(n)} \CBC_\pi(a_1,a_2,\dots, a_n). 
\end{equation}
To compute the non-vanishing cumulants we distinguish three cases.
\begin{enumerate}[(i)]
 \item 
Let $\pi\in\CI(n)$ and first assume that $\pi\in\Int(n)$. Now if $\pi<\hat{1}_n$ then $1\not\sim_\pi n$ and
moreover $1\not\sim_\rho n$ for any $\rho$ satisfying $\rho\leq\pi$.
This allows us to
replace $\tilde\omega$ by $\tilde\varphi$ to obtain 
\begin{align*}
\CBC_\pi(a_1,a_2,\dots,a_n) 
&= \sum_{\rho\leq\pi} \varphi_\rho(a_1,a_2,\dots,a_n)\,\mu(\rho,\pi) = B_\pi(a_1,a_2,\dots,a_n);  
\end{align*}
see \cite[Proposition 4.11]{Lehner:2004:cumulants1} for the last equality. 
\item 
 Let us assume next that $\pi\in\CI(n)\setminus\Int(n)$.
 This means that $\pi<\hat{1}_n$ and that $1\sim_\pi n$.
 In this case we cannot immediately replace $\omega$ by $\varphi$,
 but first must use the traciality of $\omega$ and rotate the partition
 $\pi$ into an element of $\Int(n)$.
 Indeed fix a cyclic permutation $\sigma\in\SG_n$ such that
 $\sigma\cdot\pi\in\Int(n)$. Then $1\not\sim_\pi n$ and
 also $1\not\sim_\rho n$ for any $\rho\leq \pi$ and we have
 \begin{align*}
   \CBC_\pi(a_1,a_2,\dots,a_n) 
   &=    \CBC_{\sigma\cdot\pi}(a_{\sigma^{-1}(1)},a_{\sigma^{-1}(2)},\dots,a_{\sigma^{-1}(n)})\\
   &= B_{\sigma\cdot\pi}(a_{\sigma^{-1}(1)},a_{\sigma^{-1}(2)},\dots,a_{\sigma^{-1}(n)})
   ;
 \end{align*}
 notice that the rotation cannot be reversed now because the Boolean
 cumulants are nontracial.
\item 
 Finally if $\pi=\hat{1}_n$ there is no direct formula but we infer from
 the moment-cumulant formula \eqref{eq:m-c} that
 $$
 \CBC_n (a_1,a_2,\dots,a_n) = \omega(a_1,a_2,\dots,a_n) - \sum_{\pi \in \CI(n), \pi < \hat 1_n} B_{\sigma\cdot\pi}(a_{\sigma^{-1}(1)},a_{\sigma^{-1}(2)},\dots,a_{\sigma^{-1}(n)}), 
 $$
 where for each $\pi$ an appropriate cyclic permutation $\sigma$ is chosen.
\end{enumerate}

\begin{remark}\label{rem:CB}
  Note that for univariate cumulants the rotation does not change the value
  of the cumulant and we can write
  \begin{equation}
    \label{eq:univariatecycboolcumulant}
    \omega(a^n) = c_n(a) + \sum_{\pi \in \CI(n), \pi < \hat 1_n} b_\pi(a),  
     \end{equation}
 where $b_\pi(a)= B_\pi(a,a,\dots, a)$ and $c_n(a)=\CBC_{\hat1_n}(a,a,\dots, a)$. 
\end{remark}

We can see that the definition of $c_n(a)$ in Remark \ref{rem:CB} coincides with that in \eqref{eq:CB_cumulants}. This can be confirmed from the definition and uniqueness of cumulants, but here we directly prove the formula \eqref{eq:moment-cumulant-generating-function} for $c_n(a)=\CBC_{\hat1_n}(a,a,\dots, a)$ using the recurrence relation 
\eqref{eq:univariatecycboolcumulant}. 
Decomposing $\CI(n)$ into $\Int(n)$ and $\CI(n)\setminus \Int(n)$ we obtain 
\begin{align*}
\sum_{\pi \in \CI(n), \pi \neq \hat1_n} b_\pi(a) 
&= \sum_{\pi \in \Int(n), \pi \neq \hat1_n} b_\pi(a) + \sum_{\substack{k,\ell \geq1\\ k+\ell \leq n-1}} b_{k+\ell}(a) \sum_{\sigma \in \Int(n-k-\ell)} b_\sigma(a) \\ 
&= \varphi(a^n) - b_n(a) +\sum_{\substack{k,\ell \geq1\\ k+\ell \leq n-1}} b_{k+\ell}(a)\, \varphi(a^{n-k-\ell}). 
\end{align*}
Multiplying the above by $z^n$ and taking the sum over $n$ in \eqref{eq:univariatecycboolcumulant} yields 
\begin{align*}
\MG_a(z) 
&= \CB_a(z) + M_a(z) - B_a(z) + \sum_{k,\ell \geq1} b_{k+\ell}(a)z^{k+\ell} \sum_{n\geq k+\ell+1}  \varphi(a^{n-k-\ell}) z^{n-k-\ell} \notag \\
&= \CB_a(z) + M_a(z) - B_a(z)  + (zB_a'(z) - B_a(z)) M_a(z) = \CB_a(z) +z M_a(z)B_a'(z). \label{eq:generating_cumulants}
\end{align*}

\section{Cyclic-Boolean infinite divisibility}\label{sec6}
This section is devoted to the definition and classification of infinite divisibility. Due to the lack of a precise notion of positivity (see Remark \ref{rem:positivity}), we are not able to treat general $\ast$-algebras with a state and a tracial linear functional. Hence, we give the definition of infinitely divisible distributions in the special setting of operators on Hilbert spaces where the linear functional $\omega$ is chosen to be the trace. 
\begin{definition}
Let $H$ be a Hilbert space and $\varphi$ a state on $B(H)$. An element $a \in T(H)_{\rm sa}$ is said to be cyclic-Boolean infinitely divisible if for any $n \in \N$ there exist a Hilbert space $H_n$ and a state $\varphi_n$ on $B(H_n)$ and cyclic-Boolean i.i.d.\ elements $a_1,\dots, a_n \in T(H_n)_{\rm sa}$ such that $a$ with respect to $(\varphi, \Tr_H)$ has the same distribution as $a_1 + \cdots + a_n$ with respect to $(\varphi_n, \Tr_{H_n})$. 
\end{definition}

Suppose that $a$ is a trace class selfadjoint operator and cyclic-Boolean ID. For each $n \geq2$, $a$ equals the sum of certain cyclic-Boolean iid random variables $a_{n,1},\dots, a_{n,n}$ in distribution, and let $t = 1/n$ and denote $\RC_t =\RC_{a_{n,i}}$ and $G_t=G_{a_{n,i}}$. 
Let $\{\lambda_i\}_{i\in I}$ be the set of mutually distinct eigenvalues of $a$ and $m_i$ be the multiplicity of $\lambda_i$. Setting $I_0=\{i \in I: \lambda_i\neq0\}$ we have
$$
\RC_a(z) = \sum_{i \in I_0} \frac{m_i \lambda_i}{z(z-\lambda_i)}. 
$$
Moreover, let $E_a$ be the spectral decomposition of $a$ and $p_i = \varphi(E_a(\{\lambda_i\}))\ge0$; then we have  
$$
G_a(z) = \sum_{i\in I'} \frac{p_i}{z- \lambda_i},  \qquad I'= \{i \in I: p_i >0\}. 
$$
 For later use we also set 
$$
I'_0= I' \cap I_0. 
$$
By calculus, we see that $G_a$ has a unique zero in each interval between neighboring poles and has no other zeros off the real line. Hence the set $\{\mu_j\}_{j \in J}$ of zeros of $G_a$ is contained in $\R$ and is interlacing with $\{\lambda_i\}_{i\in I'}$. 
\begin{lemma}
The factorization \label{lemma:factorization}
$$
G_a(z) = \frac{1}{z}\prod_{j\in J} \left(1-\frac{\mu_j}{z}\right) \prod_{i\in I'} \left(1-\frac{\lambda_i}{z}\right)^{-1}  
$$
holds for every $z \in \C \setminus (\{0\}\cup\{\lambda_i\}_{i\in I'})$.  
\end{lemma}
\begin{proof}
 When the set $I'$ is finite, the conclusion is easily proved since $G_a$ is a rational function. We may then assume that $I'$ and hence $I_0'$ is an infinite set. We decompose the set $\{\lambda_i\}_{i \in I_0'}$ into the positive part $\{\lambda_k^+\}_{k =1}^{n_+}$ and the negative part $\{\lambda_k^-\}_{k =1}^{n_-}$  arranged in the way
\[
\lambda_1^- < \lambda_2^- <  \cdots < 0 < \cdots < \lambda_2^+ < \lambda_1^+,   
\]
where $n_\pm \in \N \cup \{0,\infty\}$. By our assumption, $n_-$ or $n_+$ is infinity.  We rewrite the function $G_a$ into the form 
$$
G_a(z) = \frac{p}{z} + \sum_{k= 1}^{n_+} \frac{p_k^+}{z- \lambda_k^+}+ \sum_{k=1}^{n_-} \frac{p_k^-}{z- \lambda_k^-}, \qquad p = \sum_{i \in I': \lambda_i=0} p_i  \in [0,1],  
$$
where $p_k^\pm$ is weight of $\lambda_k^\pm$ with respect to $\varphi$. 
Now introduce $n_\pm(\epsilon)=\max\{k \ge1: |\lambda_k^\pm| > \epsilon \}< \infty$ 
and then the truncated function
\[
G_a^{\epsilon}(z) 
= \frac{p(\epsilon)}{z} + \sum_{k=1}^{n_+(\epsilon)} \frac{p_k^+}{z- \lambda_k^+}+ \sum_{k=1}^{n_-(\epsilon)} \frac{p_k^-}{z- \lambda_k^-},\qquad p(\epsilon)= p   + \sum_{k: |\lambda_k^+|\le \epsilon} p_k^+ +\sum_{k: |\lambda_k^-|\le \epsilon} p_k^-. 
\]
The fact that $n_-$ or $n_+$ is infinity implies that $p(\epsilon)>0$ for every $\epsilon>0$. Since $G_a^\epsilon$ is a rational function,  we have 
\begin{equation} \label{eq:truncation}
G_a^{\epsilon}(z) =\frac{1}{z}\prod_{k=1}^{n_-(\epsilon)}\left(\frac{z-\mu_k^-(\epsilon)}{z-\lambda_k^-}\right) \prod_{k=1}^{n_+(\epsilon)} \left(\frac{z-\mu_k^+(\epsilon)}{z-\lambda_k^+}\right), 
\end{equation}
where $\mu_k^\pm(\epsilon)$ is the unique zero of $G_a^{\epsilon}$ on the interval between $\lambda_{k}^\pm$ and $\lambda_{k+1}^\pm$ for $k=1,2,\dots, n_\pm(\epsilon)-1$ and $\mu_{n_\pm(\epsilon)}^\pm(\epsilon)$ is the unique zero of $G_a$ on the interval between $0$ and $\lambda_{n_\pm(\epsilon)}^\pm$.  

In order to pass to the limit in \eqref{eq:truncation}, first note that $n_\pm(\epsilon)\to n_\pm$ as $\epsilon\to0$. Since $G_a^\epsilon$ converges locally uniformly to $G_a$ on $\C\setminus\{0,\lambda_k^+,\lambda_k^-: k \ge1\}$, we conclude that $\mu_k^\pm(\epsilon)$ converges to $\mu_k^\pm$ as $\epsilon \to 0$ for each $k$. From (the product version of) Weierstrass' M-test (recall that $|\mu_k^\pm(\epsilon)|\le |\lambda_k^\pm|$) we obtain 
\[
G_a(z) =\frac{1}{z}\prod_{k=1}^{n_-}\left(\frac{z-\mu_k^-}{z-\lambda_k^-}\right) \prod_{k=1}^{n_+} \left(\frac{z-\mu_k^+}{z-\lambda_k^+}\right), 
\]
the desired formula. 
\end{proof}

Now we are able to characterize infinitely divisible measures with respect to cyclic independence.

First, notice that taking the logarithmic derivative in Lemma \ref{lemma:factorization} yields that 
$$
\frac{G_a'(z)}{G_a(z)} = -\frac{1}{z}-\sum_{i \in I'_0} \frac{\lambda_i}{z(z-\lambda_i)} + \sum_{j \in J} \frac{\mu_j}{z(z-\mu_j)}. 
$$
As before, let $t = 1/n$ and denote $\RC_t =\RC_{a_{n,i}}$ and $G_t=G_{a_{n,i}}$. Let $\{\lambda_i(t)\}_{i\in I'_0(t)}$ be the set of (mutually distinct) non-zero poles of $G_t$. Since 
$$
G_t(z) =\frac{G_a(z)}{(1-t)zG_a(z)+t}, 
$$ 
the set $\{\lambda_i(t)\}_{i\in I'_0(t)}$ is exactly the set of the zeros of $(1-t) z G_a(z) +t$. On the other hand, the set of zeros of $G_t$ is exactly the set of the zeros of $G_a$, and hence 
$$
\frac{G_t'(z)}{G_t(z)} = -\frac{1}{z}-\sum_{i \in I'_0(t)} \frac{\lambda_i(t)}{z(z-\lambda_i(t))} + \sum_{j \in J} \frac{\mu_j}{z(z-\mu_j)}. 
$$

By Theorem \ref{thm:convolution} we have
$$
\RC_a + \frac{G_a'}{G_a} + \frac{1}{z} = n \left( \RC_t + \frac{G_t'}{G_t} + \frac{1}{z}  \right) 
$$
and hence
\begin{align*}
\RC_t 
&= t \RC_a + t  \frac{G_a'}{G_a} -  \frac{G_t'}{G_t} + \frac{t-1}{z} \\
&= t \sum_{i\in I_0} \frac{m_i \lambda_i}{z(z-\lambda_i)} + t \left(-\frac{1}{z} -\sum_{i \in I'_0} \frac{\lambda_i}{z(z-\lambda_i)} + \sum_{j \in J} \frac{\mu_j}{z(z-\mu_j)}  \right) \\
&\qquad - \left(-\frac{1}{z}  -\sum_{i \in I'_0(t)} \frac{\lambda_i(t)}{z(z-\lambda_i(t))} + \sum_{j \in J} \frac{\mu_j}{z(z-\mu_j)}  \right)+\frac{t-1}{z}  \\
&= t \sum_{i\in I'_0} \frac{(m_i-1) \lambda_i}{z(z-\lambda_i)} +t \sum_{i\in I_0 \setminus I'_0} \frac{m_i \lambda_i}{z(z-\lambda_i)}  -(1-t) \sum_{j \in J} \frac{\mu_j}{z(z-\mu_j)}  +\sum_{i \in I'_0(t)} \frac{\lambda_i(t)}{z(z-\lambda_i(t))}. 
\end{align*}
Now, for each non-zero real $\alpha$ the number $\lim_{z\to \alpha} (z-\alpha)\RC_t(z)$ is non-negative, as it is a positive integer if $\alpha$ is a non-zero eigenvalue of $a_{n,1}$ and zero otherwise. 

Therefore, to cancel the negative coefficient  $-(1-t)$ above, the only possibility is that each non-zero $\mu_j$ must be a member of $\{\lambda_i\}_{i \in I_0}\cup \{\lambda_i(t)\}_{i\in I_0'(t)}$ and because of interlacing of the zeros and poles of $G_a$, one sees that $\mu_j$ can not be included in $\{\lambda_i\}_{i\in I_0'} \cup \{\lambda_i(t)\}_{i\in I_0'(t)}$ as it is a zero of $G_a$. 

It is possible that $\mu_j =\lambda_i$ for some $i\in  I_0 \setminus I_0'$.  In this case, however, for $t>0$ sufficiently small (namely, $n$ sufficiently large) the coefficient $t m_i-(1-t)$ is negative; therefore, we conclude that $\mu_j =0$ for all $j \in J$ or $J=\emptyset$, and hence $\#J=0$ or $1$.  This happens only if $\#I' = 0,1$ or $2$. 

On the other hand, for $i \in I_0'$ we have $\lim_{z\to \lambda_i} (z-\lambda_i)\RC_t(z)= t (m_i-1)$ which must be a non-negative integer for any $t=1/n$. Therefore, we conclude that $m_i=1$ for all $i\in I_0'$. For $i \in I_0 \setminus I_0'$ we have $\lim_{z\to \lambda_i} (z-\lambda_i) \RC_t(z)= t m_i$, which cannot be an integer for sufficiently small $t$, and hence $I_0'= I_0$. 

Now it remains to study the possible cases for $I_0'=I_0=0,1$ or $2$.

Case 1: $\#I_0'=\#I_0=0 $ or 1.  Then $G_a(z) = 1/(z-\alpha)$ for some $\alpha \in \R$ and $J=\emptyset$. The function $\RC_a(z)= \frac{\alpha}{z(z-\alpha)}$ and $\RC_t(z) = t\alpha /(z-t \alpha)$. 

Case 2: $\#I_0'=\#I_0=2$. Note that $G_a$ cannot have a pole at 0 because it would create a non-zero $\mu_j$.  Hence $G_a(z) = p/(z-\alpha) + (1-p)/(z-\beta)$ for some $\alpha, \beta \neq0, \alpha <\beta, 0<p<1$, and $G_a(0)= -p/\alpha - (1-p)/\beta =0$. The last condition yields the restriction that $\alpha < 0 < \beta$ and $G_a(z)=z/[(z-\alpha)(z-\beta)]$. Solving the equation $(1-t)zG_a(z)+t=0$ we obtain two solutions 
$$
\lambda_\pm(t) = \frac{t(\alpha+\beta) \pm\sqrt{t^2(\alpha+\beta)^2 -4t \alpha\beta}}{2}.  
$$
The eigenvalues can be retrieved from the formula
$$
\RC_t(z) = \frac{\lambda_+(t)}{z(z-\lambda_+(t))}+ \frac{\lambda_-(t)}{z(z-\lambda_-(t))}. 
$$

It is easy to see that the above cases are actually cyclic Boolean ID. Thus we arrive to the following.

\begin{theorem} \label{thm:CBID}
Let $H$ be a Hilbert space and $\varphi$ be a state on $T(H)$. An element $a \in T(H)_{\rm sa}$ is cyclic-Boolean ID with respect to $(\varphi,\Tr_H)$ if and only if $a$ has either 
\begin{enumerate}[\rm (i)]
\item only zero eigenvalues (that is, $a=0$), 
\item only one non-zero eigenvalue and its multiplicity is one, or
\item exactly two non-zero eigenvalues $\alpha, \beta$, their multiplicities are one, $\alpha \beta <0$ and the distribution of $a$ with respect to $\varphi$ is 
$$
\frac{-\alpha}{\beta-\alpha} \delta_{\alpha} + \frac{\beta}{\beta-\alpha} \delta_{\beta}. 
$$
\end{enumerate}
In the last case, for every $n\geq2$ an $n$-th root of $a$ has two non-zero eigenvalues $\alpha_n,\beta_n$ given as the solutions to the equation 
$$
x^2 -\frac{\alpha+\beta}{n}x + \frac{\alpha\beta}{n}=0, 
$$
and the distribution with respect to the state is 
$$
\frac{-\alpha_n}{\beta_n-\alpha_n} \delta_{\alpha_n} + \frac{\beta_n}{\beta_n-\alpha_n} \delta_{\beta_n}. 
$$ 
\end{theorem}

\begin{example}
The matrix
$$
A =\begin{pmatrix} 0 & 1 \\1 & 0\\ \end{pmatrix}
$$ 
has eigenvalues $\{[1]^1,[-1]^1\}$ and its distribution with respect to the unit vector $e_1={}^t(1,0)$ is 
 $$
 \frac{1}{2}\delta_{-1} + \frac{1}{2}\delta_1.
 $$
By Theorem \ref{thm:CBID}, $A$ is cyclic-Boolean infinitely divisible with respect to $(\langle\, \cdot \, e_1, e_1\rangle_{\C^2} , \Tr_{\C^2})$. 
\end{example}

\section{Cyclic-monotone independence}\label{sec7}

\subsection{Definition and example}

We perform an investigation for monotone independence in a spirit similar to
cyclic-Boolean independence.  To this end we start from a specific operator
model inspired by the comb product of rooted graphs in Section \ref{sec:combproduct}. 

\begin{example}
\label{ex:cyclicmonotone}  
Let $H_i, i\in \N$, be \emph{finite-dimensional} Hilbert spaces with
distinguished unit vectors $\xi_i \in H_i$ respectively. Let $P_i\colon H_i
\to H_i$ be the orthogonal projection onto $\mathbb C \xi_i$ and $\varphi_i$
be the vector state on $B(H_i)$ defined by $\xi_i$. Let $H= H_1 \otimes \cdots
\otimes H_N$, $\xi = \xi_1 \otimes \cdots \otimes \xi_N$ and $\varphi$
be the vacuum state on $B(H)$ defined by $\xi$. This is the same setting
as in Example~\ref{exa:CB} with the additional requirement of finite
dimensionality.  
Analogously to the embedding  \eqref{eq:embedbool}
we introduce another embedding of $B(H_i)$ into $B(H)$:
\begin{equation}
    \label{eq:embedmono}
\sigma_i(A) = I_{H_1} \otimes \cdots \otimes I_{H_{i-1}} \otimes A \otimes P_{i+1} \otimes \cdots \otimes P_N. 
\end{equation}
Note that this embedding does not preserve trace class and therefore
the construction is restricted to finite dimensional spaces.
It is known that the family $\{\sigma_i(B(H_i))\}_{i=1}^N$ is monotonically independent with respect to $\varphi$ ; see \cite[Theorem 8.9]{MR2316893}. 

In addition, we can compute moments with respect to the trace. Again formula
\eqref{eq:proj} is crucial: for a cyclically alternating tuple $(i_1,\dots, i_n) \in [N]^n$ and $A_k \in B(H_{i_k})$, if $p \in [n]$ is such that $i_{p-1} < i_p > i_{p+1}$ (with the conventions $i_0=i_n$ and $i_{n+1}=i_1$) then direct computations entail  
$$
\Tr_H(\sigma_{i_1}(A_1) \cdots \sigma_{i_n} (A_n)) = 
\varphi_{i_p}(A_p)\Tr_H\left[\sigma_{i_1}(A_1) \cdots\sigma_{i_{p-1}} (A_{p-1}) \sigma_{i_{p+1}}(A_{p+1}) \cdots  \sigma_{i_n} (A_n)\right], \quad n \geq2. 
$$
\end{example}

This example can be abstracted in the following way. 

\begin{definition}\label{def:cyclic_monotone} Let $(\mathcal{A},\varphi,\omega)$ be a cncps, $I$ be a toset, and $\hat I := \{-\infty\}\cup I$ be an enlargement of $I$, where $-\infty$ is the minimum of $\hat I$. An ordered family of $*$-subalgebras $\{\mathcal{A}_i\}_{i\in I}$ of $\alg{A}$ is said to be cyclic-monotone independent if 
\begin{enumerate}[\rm(i)]
\item it is monotonically independent with respect to $\varphi$, that is, for any $n\geq2$, any alternating tuple $(i_1,\dots, i_n) \in I^n$ (namely $i_1 \neq \cdots \neq i_n$) and $a_k \in \mathcal{A}_{i_k}, k=1,2,\dots,n$, if $p \in [n]$ is such that $i_{p-1} < i_p > i_{p+1}$ (with conventions $i_0=i_{n+1}=-\infty$) then 
$$
\varphi(a_1 \cdots a_n) = \varphi(a_p)\varphi(a_1 \cdots a_{p-1} a_{p+1} \cdots a_n);  
$$
\item for any $n\geq2$, cyclically alternating tuple $(i_1,\dots, i_n) \in I^n$ (namely $i_1 \neq \cdots \neq i_n \neq i_1$) and $a_k \in \mathcal{A}_{i_k}, k=1,2,\dots,n$, if $p \in [n]$ is such that $i_{p-1} < i_p > i_{p+1}$  (with different conventions $i_0=i_n$ and $i_{n+1}=i_1$) then 
$$
\omega(a_1 \cdots a_n) = 
\varphi(a_p)\omega(a_1 \cdots a_{p-1} a_{p+1} \cdots a_n). 
$$
\end{enumerate}
\end{definition}
\begin{definition} Let $(\mathcal{A},\varphi,\omega)$ be a cncps and $I$ be a toset. An ordered family of elements $\{a_i\}_{i\in I}$ of $\alg{A}$ is said to be cyclic-monotone independent if so is $\{\alg{A}_i\}_{i\in I}$, where $\alg{A}_i$ is the $\ast$-algebra generated by $a_i$ without unit. 
\end{definition}

\begin{example} Suppose that $(a,b,c)$ is cyclic-monotone independent in $(\mathcal{A},\varphi,\omega)$. Then 
\[
\varphi(ba^2bac^2b) = \varphi(c^2)\varphi(b)^3 \varphi(a^3) 
\]
and 
\[
\omega(ba^2bac^2b) = \varphi(c^2)\varphi(b) \varphi(b^2) \omega(a^3).
\]
\end{example}

\begin{remark}
Cyclic-monotone independence already appeared in the random matrix model in
\cite{CHS18} (see also \cite{MR3798863,MR4260215} and Section \ref{sec1}),
where independence was defined for a pair of $\ast$-subalgebras and only for
$\omega$. For a random matrix model for monotone independence see \cite{CebronDahlqvist}.

In \cite{CHS18} the trace functional $\omega$ is unbounded,
because it can diverge in the large dimensional limit, and therefore a domain
for $\omega$ was specified.
To avoid this problem
in the present paper we focus on finite dimensional Hilbert spaces and $\omega=\Tr$.

It should be noticed that Example \ref{ex:cyclicmonotone} does not provide an i.i.d.\ operator model even when $H_i =K$ does not depend on $i$; for $A \in B(K)$ the operators $\{\sigma_{i}(A)\}_{i=1}^N$ are identically distributed with respect to $\varphi$, but not with respect to $\omega$, because 
\[
\omega(\sigma_{i}(A)) = d^{i-1} \Tr_K(A), 
\]  
where $d = \dim (K)$. In fact we do not know of any non-trivial operator model
for cyclic monotone i.i.d.\ random variables and for this reason
we do not see any meaningful notions of cumulants and of infinitely divisible distributions. 

\end{remark}

\subsection{Cyclic-monotone convolution}
The convolution formula can be verified in ways. 
\begin{theorem}\label{thm:CM_convolution} Let $(\mathcal A,\varphi, \omega)$ be a cncps and $a,b\in \alg{A}$. Suppose that $(a,b)$ is cyclic-monotone independent. We then have
\[
\RC_{a+b}(z)  = \RC_b (z) +F_b'(z)\RC_a(F_b(z)). 
\]
\end{theorem}
\begin{proof}[Algebraic proof]
Expand $(a+b)^n$ into 
\begin{align*}
(a+b)^n
&=
b^n + \sum_{\substack{k\geq 1 \\  q_1,  q_2, \dots, q_{k+1} \geq 0 \\  q_1 +  \cdots  + q_{k+1} +k=n}} b^{q_1}a b^{q_2} a \cdots a b^{q_{k+1}}, 
\end{align*}
and applying $\omega$ yields 
\begin{align*}
\omega((a+b)^n) = \omega(b^n) + \sum_{\substack{k\geq 1 \\  q_1,  q_2, \dots, q_{k+1} \geq 0 \\  q_1 +  \cdots  + q_{k+1} +k=n}} \varphi(b^{q_1+q_{k+1}})  \varphi(b^{q_2})\cdots \varphi(b^{q_{k}}) \omega(a^k).
\end{align*}
Multiplying the above identity by $z^{-n-1}$ and taking the summation over $n$ yields 
\begin{align*}
\RC_{a+b}(z) 
&=  \RC_b (z) + \sum_{k\geq1} \sum_{q_1, \dots, q_{k+1} \geq0} \frac{\varphi(b^{q_1+q_{k+1}})}{z^{q_1+q_{k+1}}} \frac{\varphi(b^{q_2})}{z^{q_2}} \cdots \frac{\varphi(b^{q_k})}{z^{q_k}} \frac{\omega(a^k)}{z^{k+1}} \\
&=  \RC_b (z) - z^2G_b'(z) \sum_{k\geq1}[zG_b(z)]^{k-1} \frac{\omega(a^k)}{z^{k+1}} \\
&=  \RC_b (z) - \frac{G_b'(z)}{G_b(z)^2} \sum_{k\geq1} G_b(z)^{k+1} \omega(a^k)  \\
&=  \RC_b (z) +F_b'(z)\RC_a(F_b(z)). 
\end{align*}
Note here that the identity 
$$
\sum_{m \geq 0, n\geq0}^\infty \frac{\varphi(b^{m+n})}{z^{m+n}} = -z^2G_b'(z)
$$
is used above. 
\end{proof}

\begin{proof}[Analytic proof in the setting of Example \ref{ex:cyclicmonotone}: Schur complement approach]
Let $A\in B(\hilbH_1)$ and $B\in B(\hilbH_2)$ be operators with
block decompositions
$$
A =
\begin{bmatrix}
  \alpha&a'\\
  a&\bub{A}
\end{bmatrix}
\qquad \text{and}\qquad 
B =
\begin{bmatrix}
  \beta&b'\\
  b&\bub{B}
\end{bmatrix}
$$
 according to the decomposition $\hilbH_i = \mathbb{C}\xi_i \oplus \bub{\hilbH}_i$ as in \eqref{eq:Ablockdecomp}
and 
let $\sigma_1(A) = A\otimes P_2$ and $\sigma_2(B)=I_1\otimes B$ act
on $\hilbH_1\otimes \hilbH_2\simeq
\C\vac\oplus\bub{\hilbH}_1\oplus(\hilbH_1\otimes\bub{\hilbH}_2)$
according to Example~\ref{ex:cyclicmonotone}, i.e., if we denote by $\eta_1:\bub{\hilbH}_1\to\hilbH_1$ the embedding
and  $\eta_1^*:\hilbH_1\to\bub{\hilbH}_1$ the projection, then
\begin{align*}
  \sigma_1(A)\vac &= \alpha\xi\oplus a\oplus 0
                    & \sigma_2(B)\vac &= \beta\vac\oplus0\oplus(\vac_1\otimes b)\\
  \sigma_1(A)\bub{h}_1 &= (a'\bub{h}_1)\vac\oplus\bub{A}\bub{h}_1\oplus 0
                    & \sigma_2(B)\bub{h}_1 &= 0\oplus\beta\bub{h}_1\oplus(\eta_1(\bub{h}_1)\otimes b)\\
    \sigma_1(A)(h_1\otimes\bub{h}_2) &= 0
                    & \sigma_2(B)(h_1\otimes\bub{h}_2) &= (b'\bub{h}_2)(\vac_1^*h_1)\vac\oplus(b'\bub{h}_2)\eta_1^*(h_1)\oplus(h_1\otimes\bub{B}\bub{h}_2)
\end{align*}
and we obtain the block decompositions
$$
\sigma_1(A) =
\begin{bmatrix}
  \alpha & a' & 0\\
  a&\bub{A}&0\\
  0&0&0
\end{bmatrix}, 
\qquad
\sigma_2(B) =
\begin{bmatrix}
  \beta & 0 &  \vac_1^*\otimes b'\\
  0 &\beta \bub{I}_1& \eta_1^*\otimes b'\\
  \vac_1\otimes b&\eta_1\otimes b&I_1\otimes\bub{B}
\end{bmatrix}
$$
and together
$$
\sigma_1(A) + \sigma_2(B) =
\begin{bmatrix}
  \alpha+  \beta & a' & \vac_1^*\otimes b'\\
 a&\bub{A}+\beta \bub{I}_1& \eta_1^*\otimes b'\\
  \vac_1\otimes b&\eta_1\otimes b&I_1\otimes\bub{B}
\end{bmatrix}. 
$$
We compute the resolvent
\begin{equation}
  \label{eq:z-sigma1A-sigma2B}
(z-\sigma_1(A)-\sigma_2(B))^{-1}
=
\left[
  \begin{array}{c|cc}
    z-\alpha-\beta &- a' & -\vac_1^*\otimes b'\\
    \midrule
    -a&(z-\beta) \bub{I}_1-\bub{A}& -\eta_1^*\otimes b'\\
    -\vac_1\otimes b&-\eta_1\otimes b&I_1\otimes(z\bub{I}_2-\bub{B})
  \end{array}
\right]^{-1}
\end{equation}
via the Schur complement.
To this end we first compute the lower resolvent
\begin{equation}
  \label{eq:lowerresolvent}
  L^{-1}=
\begin{bmatrix}
      (z-\beta) \bub{I}_1-\bub{A}& -\eta_1^*\otimes b'\\
    -\eta_1\otimes b&I_1\otimes(z\bub{I}_2-\bub{B})
  \end{bmatrix}^{-1}
\end{equation}
on $\bub{\hilbH}_1\oplus(\hilbH_1\otimes \bub{\hilbH}_2)$.
The corresponding Schur complement of $L$ is
\begin{align*}
  S_L &= (z-\beta)\bub{I}_1-\bub{A} - (\eta_1^*\otimes
      b')(I_1\otimes(z\bub{I}_2-\bub{B})^{-1})(\eta_1\otimes b)\\
    &= (z-\beta)\bub{I}_1-\bub{A} - \eta_1^*I_1\eta_1\otimes
      b'(z\bub{I}_2-\bub{B})^{-1}b\\
      &= F_B(z)\bub{I}_1-\bub{A}
        .
\end{align*}
If we denote by
$\resolv_{\bub{A}}(F_B(z))=(F_B(z)\bub{I}_1-\bub{A})^{-1}$
and $\resolv_{\bub{B}}(z)=(z\bub{I}_2-\bub{B})^{-1}$
the resolvents of $\bub{A}$ and $\bub{B}$, respectively,
then with the help of Banachiewicz' formula  \eqref{eq:Banachiewicz}
the resolvent  \eqref{eq:lowerresolvent} can be written as
\begin{align*}
  L^{-1}
  &=
 \begin{bmatrix}
   S_L^{-1} & S_L^{-1}(\eta_1^*\otimes b')(I_1\otimes \resolv_{\bub{B}}(z)  )  \\
    I_1\otimes\resolv_{\bub{B}}(z)(\eta_1\otimes b)S_L^{-1} &
    I_1\otimes \resolv_{\bub{B}}(z)+
    (I_1\otimes\resolv_{\bub{B}}(z))(\eta_1\otimes b)S_L^{-1}(\eta_1^*\otimes b')(I_1\otimes\resolv_{\bub{B}}(z))
  \end{bmatrix}
  \\
  &=
    \begin{bmatrix}
    \resolv_{\bub{A}}(F_B(z))
    & \resolv_{\bub{A}}(F_B(z))\eta_1^*\otimes b'\resolv_{\bub{B}}(z)\\
    \eta_1\resolv_{\bub{A}}(F_B(z))\otimes\resolv_{\bub{B}}(z)b
    & I_1\otimes\resolv_{\bub{B}}(z)
      +\eta_1\resolv_{\bub{A}}(F_B(z))\eta_1^*\otimes \resolv_{\bub{B}}(z)bb'\resolv_{\bub{B}}(z)
  \end{bmatrix}. 
\end{align*}
Now we plug $L^{-1}$  into Banachiewicz' formula \eqref{eq:Banachiewicz} for \eqref{eq:z-sigma1A-sigma2B}:
$$
(z-\sigma_1(A)-\sigma_2(B))^{-1} = 
\left[
  \begin{array}{c|c}
    S^{-1} & S^{-1}
             \begin{bmatrix}
               a'&\vac_1^*\otimes b'
             \end{bmatrix}L^{-1}\\
    \midrule
    L^{-1}
    \begin{bmatrix}
      a\\ \vac_1\otimes b
    \end{bmatrix}
    S^{-1}
    &L^{-1}+L^{-1}
    \begin{bmatrix}
      a\\ \vac_1\otimes b
    \end{bmatrix}
    S^{-1}
    \begin{bmatrix}
      a'&\vac_1^*\otimes b'
    \end{bmatrix}L^{-1}
  \end{array}
\right]. 
$$
After some cancellations the Schur complement evaluates to
\begin{align*}
  S&=F_{\sigma_1(A)+\sigma_2(B)}(z)\\
  &= z -\alpha-\beta -
    \begin{bmatrix}
-a' & -\vac_1^*\otimes b'
\end{bmatrix}
       L^{-1}
  \begin{bmatrix}
    -a\\ -\vac_1\otimes b
  \end{bmatrix}
  \\
  &= z-\alpha-\beta-a'(F_B(z)\bub{I}_1-\bub{A})^{-1}a
    -\vac_1^*\vac_1b'(z\bub{I}_2-\bub{B})^{-1}b\\
  &= F_B(z)  - \alpha-a'(F_B(z)\bub{I}_1-\bub{A})^{-1}a
  \\
  &= F_A(F_B(z))
\end{align*}
where we used  the Schur complement representation
\eqref{eq:GA=banach}.
Finally the resolvent is
\begin{multline*}
   (z-\sigma_1(A)-\sigma_2(B))^{-1}\\
  = 
  G_A(F_B(z))\left[
  \begin{array}{c|c}
    1 & a'\resolv_{\bub{A}}(F_B(z)) \qquad [a'\resolv_{\bub{A}}(F_B(z))\eta_1^*+\vac_1^*]\otimes b'\resolv_{\bub{B}}(z)\\
    \midrule
    \begin{matrix}
    \resolv_{\bub{A}}(F_B(z))a \\
    [\eta_1\resolv_{\bub{A}}(F_B(z))a+\vac_1]\otimes\resolv_{\bub{B}}(z)b
  \end{matrix}
    & 
      F_A(F_B(z)) L^{-1} +  L_2
  \end{array}
\right]
\end{multline*}
where
$$
 L_2 = \begin{bmatrix}
    T &
    T\eta_1^* +\resolv_{\bub{A}}(F_B(z))a\vac_1^*\otimes b'\resolv_{\bub{B}}(z)
    \\
    \eta_1T+\vac_1a'\resolv_{\bub{A}}(F_B(z))\otimes \resolv_{\bub{B}}(z)b
    & 
    \bigl(
    \eta_1T\eta_1^* + \eta_1\resolv_{\bub{A}}(F_B(z))a\vac_1^* +
    \vac_1a'\resolv_{\bub{A}}(F_B(z))\eta_1^* + \vac_1\vac_1^*
    \bigr)\otimes\resolv_{\bub{B}}(z)
  \end{bmatrix}
$$
with
$$
T=    (F_B(z) -\bub{A})^{-1}aa'(F_B(z) -\bub{A})^{-1}
.
$$
Finally the trace of the resolvent evaluates to
\begin{align*}
 \Cauchy_{\sigma_1(A)+\sigma_2(B)}(z) 
   &= G_A(F_B(z)) + \Tr(L^{-1}) + G_A(F_B(z))\Tr(L_2)\\
   &= G_A(F_B(z)) + \Cauchy_{\bub{A}}(F_B(z)) +
     \Tr(I_1)\Cauchy_{\bub{B}}(z) + \Tr[\eta_1(F_B(z)-\bub{A})^{-1}\eta_1^*]   \\&  \phantom{==}
      +\Tr[(z-\bub{B})^{-1}bb'(z-\bub{B})^{-1}]   + G_A(F_B(z))
        \Tr[T (F_B(z)-\bub{A})^{-1}aa'(F_B(z)-\bub{A})^{-1}] \\ 
    &\phantom{==}  +G_A(F_B(z)) \Tr(\eta_1T \eta_1^* +\vac_1\vac_1^*)
          \Tr[(z-\bub{B})^{-1}bb'(z-\bub{B})^{-1}]
  \\
  &= G_A(F_B(z)) + \Cauchy_{\bub{A}}(F_B(z)) + d_1 \Cauchy_{\bub{B}}(z) +
    \Cauchy_{\bub{A}}(F_B(z))(F_B'(z)-1)
    \\&\phantom{==}
        + G_A(F_B(z))[F_A'(F_B(z))-1 + F_A'(F_B(z))(F_B'(z)-1)]
  \\
  &= d_1\Cauchy_{\bub{B}}(z) + \Cauchy_A(F_B(z))F_B'(z). 
\end{align*}
\end{proof}

We note here that another equivalent convolution formula can be given in terms of
 \[
 \KK_a(z) = -\sum_{n\geq1} \frac{\omega(a^n)}{ n z^n}.
 \]
The convolution formula in Theorem \ref{thm:CM_convolution} then reads
$
\KK_{a+b}' =  \KK_b' + (\KK_a \circ F_b)'
$
and hence 
\begin{equation*}
\KK_{a+b}(z) =  \KK_b(z) + \KK_a (F_b(z)). 
\end{equation*}

\subsection{Limit theorem}

In the setting of Example \ref{ex:cyclicmonotone}, let $H_i$ be the same Hilbert space $K$ with $d = \dim(K) \in \{2,3,4,\dots \}$ and with a distinguished unit vector $\xi$ and let $a^{(i)}=\sigma_i(a)$ for some $a \in B(K)_{\rm sa}$. Let $\omega$ be the trace on $H$, $\Tr$ be the trace on $K$ and $\psi$ be the vector state on $K$ determined by $\xi$.  

Our main object in this section is the sum
\begin{equation}\label{eq:iid_sum}
b_N = a^{(1)} + \dots + a^{(N)}. 
\end{equation}  
In order to see the convergence of trace moments of $b_N$ we start from some examples. 
Since $\Tr(I_K)=d$, we obtain 
$$
\omega(b_N) = \Tr(a) + d \Tr(a) + \cdots + d^{N-1} \Tr(a) = [N]_d \Tr(a), 
$$
where $[N]_d=1+d+d^2+\dots +d^{N-1}$, and 
\begin{align*}
\omega(b_N^2) 
&= \sum_{i<j} \omega(a^{(i)}a^{(j)}) + \sum_{i>j} \omega(a^{(i)}a^{(j)}) + \sum_{i} \omega((a^{(i)})^2) \\
&= \sum_{i<j} \omega(a^{(i)})\varphi(a^{(j)}) + \sum_{i>j} \varphi(a^{(i)} ) \omega(a^{(j)}) + [N]_d\Tr(a^2) \\
&= \sum_{j=2}^N \frac{d^{j-1}-1}{d-1} \Tr(a) \psi(a) + \sum_{i=2}^N \frac{d^{i-1}-1}{d-1} \Tr(a) \psi(a) + [N]_d\Tr(a^2) \\
&=  \frac{2}{d-1}\left[\frac{d(d^{N-1}-1)}{d-1}-(N-1) \right] \Tr(a) \psi(a) + [N]_d\Tr(a^2) \\ 
&=  \frac{2}{d-1}([N]_d -N ) \Tr(a) \psi(a) + [N]_d\Tr(a^2).  
\end{align*}

A similar computation yields that 
\begin{align*}
\omega(b_N^3) 
&= 6 \left[  \frac{[N]_d -N}{(d-1)^2} - \frac{N(N-1)}{2(d-1)}\right] \Tr(a)\psi(a)^2 +  \frac{3}{d-1}([N]_d -N ) \Tr(a^2) \psi(a) \\
&\quad +  \frac{3}{d-1}([N]_d -N ) \Tr(a) \psi(a^2)+ [N]_d \Tr(a^3).  
\end{align*}
Therefore, the normalized traces converge as $N\to \infty$ without rescaling of $b_N$: 
\begin{align*}
d^{-N}\omega(b_N) 
&\to \frac{\Tr(a)}{d-1}, \\
d^{-N}\omega(b_N^2) 
&\to \frac{2}{(d-1)^2}  \Tr(a) \psi(a) + \frac{1}{d-1}\Tr(a^2), \\
d^{-N}\omega(b_N^3) 
&\to  \frac{6}{(d-1)^3} \Tr(a)\psi(a)^2  +  \frac{3}{(d-1)^2} \Tr(a^2) \psi(a) +  \frac{3}{(d-1)^2} \Tr(a) \psi(a^2) +  \frac{1}{d-1} \Tr(a^3). 
\end{align*}

In order to describe the general situation, we need some concepts on ordered  set partitions. 

\begin{definition} Let $k \in \N$. 
\begin{enumerate}[\rm(i)]
\item An ordered set partition of $[k]$ is a tuple $\pi=(B_1,B_2,\dots, B_p)$ of subsets of $[k]$ such that $\{B_1,\dots, B_p\}$ is a set partition of $[k]$; that is, $B_1, \dots, B_p$ are non-empty and mutually disjoint subsets of $[k]$, and their union is $[k]$. The length $p$ of $\pi$ is denoted by $|\pi|$. The set of the ordered set partitions of $[k]$ is denoted by $\OP(k)$. 

\item For a tuple $\mathbf i =(i_1, \dots, i_k) \in \N^k$, the ordered kernel set partition $\ker(\mathbf i) \in \OP(k)$ is defined as follows: first, pick the smallest value $p_1$  among $i_1,\dots, i_k$ and then define the subset $B_1=\{ j \in [k]: i_j=p_1\}$; 
secondly, pick the second smallest value $p_2$ among $i_1,\dots, i_k$ and define the subset $B_2=\{j \in [k]: i_j=p_2\}$; continuing this procedure until the end we arrive at an ordered set partition $(B_1,B_2,\dots)$, which is denoted by $\ker(\mathbf i)$. 

\end{enumerate}
\end{definition}

\begin{example}
\begin{align*}\label{eq:osp}
\ker(6,3,2,3,6) &= (\{3\},\{2,4\},\{1,5\}), \\
\ker(2,7,4,7,4,2,4)&=(\{1,6\},\{3,5,7\},\{2,4\}). 
\end{align*}
For further information on ordered (kernel) set partitions  the reader is referred to \cite{HasebeLehner:cumulants5}. 

\end{example} 
 For an ordered set partition $\pi$ of $[k]$ there exists a unique \emph{packed
   word}, i.e., a tuple $\mathbf i(\pi)= (i_1(\pi),\dots, i_k(\pi)) \in [|\pi|]^k$ such that $\pi = \ker(\mathbf i(\pi))$. Using this tuple we define $\omega(\pi)$ to be $\omega(a^{(i_1(\pi))} \cdots a^{(i_k(\pi))})$. 
 
\begin{example} 
 If  $\pi=(\{1,3\},\{2\})$ then $\mathbf i(\pi)=(1,2,1)$ and 
\[
\omega(\pi) = \omega(a^{(1)}a^{(2)}a^{(1)})=\varphi(a)\Tr(a^2).
\] 
 If $\pi=(\{3\},\{2,4,6\},\{1,5\})$ then $\mathbf i (\pi) = (3,2,1,2,3,2)$ and 
 \[
\omega(\pi) = \omega(a^{(3)}a^{(2)}a^{(1)}a^{(2)}a^{(3)}a^{(2)})=\varphi(a)\varphi(a) \varphi(a^3)\Tr(a).
\] 
 \end{example}
With those notions, we have
\begin{align*}
\omega(b_N^k) 
&= \sum_{\pi \in \OP(k)} \sum_{\substack{\mathbf i=(i_1\dots, i_k) \in [N]^k\\ \ker(\mathbf i) = \pi}} \omega(a^{(i_1)} \cdots a^{(i_k)}) \\
&=  \sum_{\pi \in \OP(k)} \sum_{\substack{\mathbf i=(i_1\dots, i_k) \in [N]^k\\ \ker(\mathbf i) = \pi}} d^{\min\{i_1,\dots, i_k\}-1} \omega(\pi) \\
&=  \sum_{\pi \in \OP(k)} \alpha_{|\pi|}(d,N) \omega(\pi),  
\end{align*}
where 
\[\alpha_p(d,N):= \sum_{\substack{(j_1,\dots, j_p) \in [N]^p \\ j_1 < \cdots <j_p}} d^{j_1-1},  \quad N \ge p;  \qquad \alpha_p(d,N) :=0, \quad 0\le N<p.  
\]

In order to investigate the asymptotics of $\omega(b_N^k)$ it suffices to understand the function $\alpha_{k}(d, N)$. 
\begin{lemma}\label{lem:alpha}
For each $k \in \mathbb N$ there exists a polynomial $P_k$ in two variables such that 
\[
\alpha_k(d,N) = \frac{d^N}{(d-1)^k} + P_k((d-1)^{-1},N), \qquad d\geq 2, N \geq0.  
\]
\end{lemma}
\begin{proof} 
The proof goes by induction on $k$. For $k=1$,  $\alpha_1(d,N)=(d^N-1)/(d-1)$, and hence $P_1(x,y) = -x$. 
For general $k\ge2$ we proceed as 
\begin{align*}
\alpha_k(d,N)
&= \sum_{j=1}^N\sum_{\substack{(j_1,\dots, j_{k-1}) \in [j-1]^{k-1} \\  j_1 < \cdots < j_{k-1} }} d^{j_1-1} = \sum_{j=1}^N \alpha_{k-1}(d,j-1) \notag \\
&=  \sum_{j=1}^N \left[\frac{d^{j-1}}{(d-1)^{k-1}} + P_{k-1}((d-1)^{-1},j-1)\right] \\
&=\frac{d^N}{(d-1)^{k}} - \frac{1}{(d-1)^{k}}+ \sum_{j=1}^N P_{k-1}((d-1)^{-1},j-1), \qquad \qquad N\ge 1. 
\end{align*}
By Faulhaber's formula and induction hypothesis, there is a polynomial $Q_k(x,y)$ such that $Q_k(x,0)=0$ and 
\[
Q_k(x, N) = \sum_{j=1}^N P_{k-1}(x,j-1), \qquad N\ge 1, 
\]
which implies the desired formula for $N\ge 1$ by taking $P_k(x,y)=-x^k +Q_k(x,y)$. Since $Q_k(x,0)=0$ the formula holds for $N=0$ as well.  
\end{proof}

By Lemma \ref{lem:alpha} we obtain the limit 
\begin{equation*}
\lim_{N\to\infty} d^{-N} \alpha_k(d,N) = \frac{1}{(d-1)^k}
\end{equation*}
and conclude the following. 
\begin{theorem} In the setting above we have
\begin{equation}\label{eq:limit}
\lim_{N\to\infty} d^{-N}\omega(b_N^k)  =  \sum_{\pi \in \OP(k)} \frac{\omega(\pi)}{(d-1)^{|\pi|}}. 
\end{equation}
\end{theorem}
Thus the empirical eigenvalue distributions of $b_N$ converge (in the sense of
moments) to a probability measure whose $k$-th moment is the above limit. Of
course the empirical eigenvalue distributions of the rescaled sum $N^{-1/2}b_N$
converge weakly to $\delta_0$, which means that the number of eigenvalues of
$N^{-1/2}b_N$ outside a fixed neighborhood of $0$ is of the order
$o(d^N)$. Combining this with the monotone CLT, which asserts that vacuum spectral distribution of $N^{-1/2} b_N$ weakly converges an arcsine distribution, it turns out that the vacuum vector captures a relatively small number of eigenvalues of $N^{-1/2}b_N$ that lie outside the neighborhood of 0.

The limit moments \eqref{eq:limit} depend on a lot of information about trace
and vacuum moments of the original matrix $a$. This is in sharp contrast with the fact that if $\psi(a)=0$ and $\psi(a^2)=1$ then the distribution of the rescaled sum $N^{-1/2}b_N$  with respect to the vacuum state $\varphi$ converges weakly to the same arcsine law. 

We come back to the original model of comb product graphs in Section \ref{sec:combproduct} (cf: Example \ref{ex:cyclicmonotone}), and compute the limit empirical eigenvalue distribution of the adjacency matrix of the iterated comb product of the complete graph $\bK_2$. Even for this simplest graph, the limit moments \eqref{eq:limit} are not explicit; they only satisfy a recurrence relation. Fortunately, we can describe the limit distribution with the help of work of Smyth \cite{Smy80}, who defined a  distribution function $L_+\colon [0,\infty)\to [0,1)$ (denoted as $F$ therein) characterized by the property that $L_+$ is strictly increasing, $L_+(0)=0$ and 
$$
|2L_+(x)-1|=L_+(|x-x^{-1}|),\qquad x>0. 
$$
 Let $\lambda_+$ be the distribution associated with $L_+$ and $\lambda$ be the symmetrization of $\lambda_+$. It is known that $L_+$ is continuous and hence $\lambda$ has no atoms. 

\begin{theorem}
Let $A_N$ be the adjacency matrix of the $N$-fold comb product of $(\bK_2,o)$ with itself. Then the empirical eigenvalue distribution of $A_N$ converges weakly to $\lambda$ as $N\to\infty$.  
\end{theorem}
\begin{proof} In the notation of this section, we are dealing with 
\[K=\C^2, \qquad a = \begin{pmatrix} 0 & 1 \\ 1 & 0 \end{pmatrix}, \quad {\rm and} \quad \xi = \begin{pmatrix} 1 \\ 0 \end{pmatrix}.
\]
 As already verified, the limiting $p$-th moment of $b_N$ is described by 
\begin{equation}\label{eq:moment}
 \sum_{\pi \in \OP(p)}\omega(\pi),  
\end{equation}
where $\omega(\pi)$ in \eqref{eq:limit} is determined by the sequence
$\{\omega(a^p)\}_{p\geq0} = \{2,0,2,0,\dots\} = \{2 \psi(a^p)\}_{p\geq0}$. It
is easy to see that $\omega(\pi)=0$ for all $\pi \in \OP(p)$ if $p$ is odd, and
hence all odd moments vanish. We will compute the numbers 
\begin{equation*}
\gamma_{n,k}:= \sum_{\pi \in \OP(2n), |\pi|=k}\omega(\pi). 
\end{equation*}
The very definition of $\omega(\pi)$ shows that $\gamma_{n,1}=2$ and $\omega(\pi)$ is either 0 or 2.  
Below we identify $[p]$ with $\Z_p$ regarded as points on a circle.
Let $p \geq2$, then a \emph{maximal arc} in a subset $B \subset \Z_p$ is a
maximal cyclic interval $I\subseteq \Z_p$  contained in $B$.

Any subset $B \subset \Z_p$ is a union of maximal arcs of $B$ in $\Z_p$ (see Fig.\ \ref{fig1}). 
This notion is important since for $\pi=(B_1,\dots, B_k) \in \OP(2n)$ the factorization 
\begin{equation}\label{rec}
\omega(\pi) = \omega(\pi|_{[2n]\setminus B_k}) \prod_{I: \text{maximal arc  of $B_k$ in $\Z_{2n}$}} \psi(a^{\# I}) 
\end{equation}
holds. 
Observe from the repeated use of \eqref{rec} that $\omega(\pi)=2$ if and only if
\begin{enumerate}[(a)]
\item\label{item:maximal} each maximal arc of $B_k$ in $[2n] \simeq \Z_{2n}$ has even size, 
\item each maximal arc of $B_{i}$ in $[2n] \setminus (B_{i+1} \cup \cdots \cup B_k) \simeq \Z_{2n - \sum_{j=i+1}^{k}\# B_{j}}$ has even size for all $i=1,2,\dots, k-1$. 
\end{enumerate}
 Note that these conditions imply that all $B_i$ have even size. Moreover, since $B_i$ are not empty, we must have $2 \leq \# B_i \leq 2n- 2k+2$ for all $i \in[k]$.

To find a recursive formula for $\gamma_{n,k}$, we count the number $\delta_{n,m}$ of all subsets $B_k \subset [2n]$ satisfying \eqref{item:maximal} and with $\#B_k=2m$ for each $1\leq k \leq n$ and $1 \leq m \leq n-k+1$. As in Fig.\ \ref{fig1}, two neighboring elements $\bullet$ of $B_k$ can be joined to a single element $\star$ and the elements on the circle can be cut between $1$ and $2n$ and be opened to a line, so the problem comes to counting the number of arranging $m$ elements $\star$ and $2n-2m$ elements $\circ$ on one line; however, subsets like $B_k=\{1,4,5,8\}$ does not correspond to such a line arrangement, so we adjust such a case by rotating the circle to the left as in Fig.\ \ref{fig2}. Therefore, a line arrangement of $m$ elements $\star$ and $2n-2m$ elements $\circ$ corresponds to a single $B_k$ if it ends with $\circ$, while it corresponds to two $B_k$'s if it ends with $\star$. Altogether, we arrive at
$$
\delta_{n,m} =  \binom{m+2n-2m-1}{m} +2\binom{m+2n-2m-1}{m-1} = \binom{2n-m}{m} \frac{2n}{2n-m}. 
$$
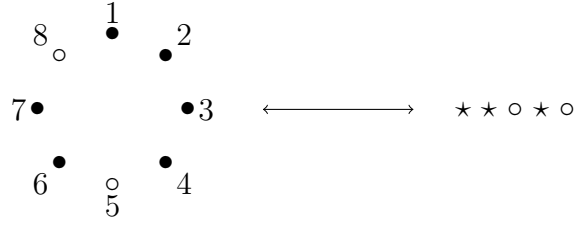
\begin{figure}[bpt]
\begin{center}
\begin{tikzpicture}[scale = 1]
\node at (1,0) [rectangle]  {$\bullet$};
\node at (1,0) [right]  {$3$};
\node at (0.707,0.707) [rectangle]  {$\bullet$};  
\node at (0.707,0.707) [above right]  {$2$};  
\node at (0,1) [rectangle]  {$\bullet$};
\node at (0,1) [above]  {$1$};
\node at (-0.707,0.707) [rectangle]  {$\circ$};  
\node at (-0.707,0.707) [above left]  {$8$};  
\node at (-1,0) [rectangle]  {$\bullet$};
\node at (-1,0) [left]  {$7$};
\node at (-0.707,-0.707) [rectangle]  {$\bullet$};  
\node at (-0.707,-0.707) [below left]  {$6$};  
\node at (0,-1) [rectangle]  {$\circ$}; 
\node at (0,-1) [below]  {$5$};  
\node at (0.707,-0.707) [rectangle]  {$\bullet$};  
\node at (0.707,-0.707) [below right]  {$4$};
---------------------------
\draw[<->] (2,0) -- (4,0) node[right] {$\quad\star$ $\star$ $\circ$ $\star$ $\circ$}; 
\end{tikzpicture}
\caption{$B_k=\{1,2,3,4,6,7\} \subset [8]$ and its unfolded line, cut between 1 and 8. The maximal arcs of $B_k$ in $[8]$ are $\{1,2,3,4\}$ and $\{6,7\}$.}\label{fig1}
\end{center}
\end{figure}

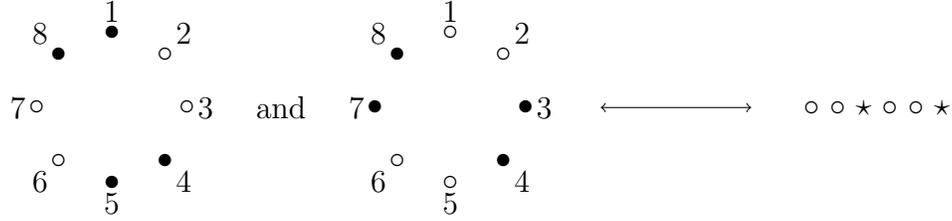
\begin{figure}
\begin{center}
\begin{tikzpicture}[scale = 1]
\node at (5-4.5,0) [rectangle]  {$\circ$};
\node at (5-4.5,0) [right]  {$3$};
\node at (4.707-4.5,0.707) [rectangle]  {$\circ$};  
\node at (4.707-4.5,0.707) [above right]  {$2$};  
\node at (4-4.5,1) [rectangle]  {$\bullet$};
\node at (4-4.5,1) [above]  {$1$};
\node at (4-0.707-4.5,0.707) [rectangle]  {$\bullet$};  
\node at (4-0.707-4.5,0.707) [above left]  {$8$};  
\node at (4-1-4.5,0) [rectangle]  {$\circ$};
\node at (4-1-4.5,0) [left]  {$7$};
\node at (4-0.707-4.5,-0.707) [rectangle]  {$\circ$};  
\node at (4-0.707-4.5,-0.707) [below left]  {$6$};  
\node at (4-4.5,-1) [rectangle]  {$\bullet$}; 
\node at (4-4.5,-1) [below]  {$5$};  
\node at (4.707-4.5,-0.707) [rectangle]  {$\bullet$};  
\node at (4.707-4.5,-0.707) [below right]  {$4$};
----
\node at (2-0.25,0) {and};
----
\node at (1-0.5+4.5,0) [rectangle]  {$\bullet$};
\node at (1-0.5+4.5,0) [right]  {$3$};
\node at (0.707-0.5+4.5,0.707) [rectangle]  {$\circ$};  
\node at (0.707-0.5+4.5,0.707) [above right]  {$2$};  
\node at (0-0.5+4.5,1) [rectangle]  {$\circ$};
\node at (0-0.5+4.5,1) [above]  {$1$};
\node at (-0.707-0.5+4.5,0.707) [rectangle]  {$\bullet$};  
\node at (-0.707-0.5+4.5,0.707) [above left]  {$8$};  
\node at (-1-0.5+4.5,0) [rectangle]  {$\bullet$};
\node at (-1-0.5+4.5,0) [left]  {$7$};
\node at (-0.707-0.5+4.5,-0.707) [rectangle]  {$\circ$};  
\node at (-0.707-0.5+4.5,-0.707) [below left]  {$6$};  
\node at (0-0.5+4.5,-1) [rectangle]  {$\circ$}; 
\node at (0-0.5+4.5,-1) [below]  {$5$};  
\node at (0.707-0.5+4.5,-0.707) [rectangle]  {$\bullet$};  
\node at (0.707-0.5+4.5,-0.707) [below right]  {$4$};
-----------------
\draw[<->] (6,0) -- (8,0) node[right] {$\quad$ $\circ$ $\circ$ $\star$ $\circ$ $\circ$ $\star$}; 
\end{tikzpicture}
\caption{$\{1,4,5,8\}\subset [8]$ and its rotation to the left $\{3,4,7,8\} $}\label{fig2}
\end{center}
\end{figure}

From this counting, \eqref{rec} gives the recursive formulas
\begin{align}
\gamma_{n,k} 
&= \sum_{\substack{B_k \subset [2n] \\ 2 \leq \# B_k \leq 2n-2k+1 \\ \text{$B_k$ satisfies \eqref{item:maximal}}}} \gamma_{n- \frac{1}{2}\# B_k,k-1 }  = \sum_{m=1}^{n-k+1} \delta_{n,m} \gamma_{n-m,k-1} \notag \\
&= \sum_{m=1}^{n-k+1} \binom{2n-m}{m} \frac{2n}{2n-m} \gamma_{n-m,k-1} =  \sum_{\ell=k-1}^{n-1} \binom{n+\ell}{n-\ell} \frac{2n}{n+\ell} \gamma_{\ell,k-1}, \qquad n \geq k \geq 2,   \label{rec2}
\end{align}
with $\gamma_{n,1}=2$ for $n\geq1$. 
Let 
\[
\beta_n := \sum_{k=1}^n \gamma_{n,k}, \qquad n\geq1; \qquad \beta_0:=1.
\] 
Note that $\beta_n$ is the limiting moment \eqref{eq:moment} of order $p=2n$. The recursive formula \eqref{rec2} implies  
\begin{equation}\label{eq:recursion}
\beta_n = \sum_{\ell=0}^{n-1} \binom{n+\ell}{n-\ell} \frac{2n}{n+\ell}  \beta_\ell, \qquad n\geq1, 
\end{equation}
and thus $\beta_n$ satisfies the recurrence in \cite[Theorem 6]{Smy84} and hence coincides with the $2n$-th moment of $\lambda$. Some examples are:  $\{\beta_n\}_{n=1}^7 =\{ 2, 10, 80, 874, 12092, 202384, 3973580\}$.

It then suffices to show the determinacy of the moment problem to conclude the weak convergence. According to Lemma \ref{lem_carleman} below, one easily sees Carleman's condition 
\[
\sum_{n\geq1} \beta_n^{-\frac{1}{2n}} =\infty, 
\]
which shows that the moment problem for \eqref{eq:moment} is determinate. 
\end{proof}

\begin{lemma}\label{lem_carleman} Let $\{\beta_n\}_{n\geq0}$ be the sequence defined by \eqref{eq:recursion} with $\beta_0=1$. There exists $C>0$ such that 
the inequality $\beta_n \leq (C n)^{2n}$ holds for all $n\geq1$. 
\end{lemma}
\begin{remark}
The proof below shows that $C=11$ suffices. Moreover, according to OEIS A048286, a more precise asymptotics $\beta_n \sim c (2/(e \log 2))^n  n^{n + \frac{1}{2} -\frac{\log 2}{4}}$ holds,  where $c = 1.6463...$. 
\end{remark}
\begin{proof}
 We proceed by induction using the recursion formula \eqref{eq:recursion} and Stirling's formula 
 \[
 \sqrt{2\pi}  n^{n+1/2} e^{-n} \leq n! \leq  e  n^{n+1/2} e^{-n},\qquad n\geq1.
 \]
Let $n\geq2$. We adopt the notation $0^0=1$. Assuming the desired inequality holds until $n-1$ for some constant $C>1$, one has 
 \begin{align*}
 \beta_n 
 &= 2+\sum_{\ell=1}^{n-1} \binom{n+\ell}{n-\ell} \frac{2n}{n+\ell} \beta_\ell 
 \leq 2+ 2\sum_{\ell=1}^{n-1} \binom{n+\ell}{n-\ell} (C\ell)^{2\ell} \\
 & \leq 2+ \frac{e}{\pi\sqrt{2}}\sum_{\ell=1}^{n-1}  \frac{(n+\ell)^{n+\ell +1/2}}{\sqrt{\ell} 2^{2\ell} (n-\ell)^{n-\ell +1/2}} C^{2\ell} 
  \leq 2+ \frac{e}{\pi\sqrt{2}}\sum_{\ell=1}^{n-1}  \frac{(n+\ell)^{n+\ell +1/2}}{ 2^{2\ell} (n-\ell)^{n-\ell +1/2}} C^{2\ell}. 
 \end{align*}
 We split the sum into the two parts $1 \leq \ell \leq n-4$ and $n-3 \leq \ell \leq n-1$ (the arguments below are valid even for $2 \leq n \leq 4$ by setting the irrelevant terms to be 0). The first part is estimated as
 \begin{align*}
 \sum_{\ell=1}^{n-4}  \frac{(n+\ell)^{n+\ell +1/2}}{ 2^{2\ell} (n-\ell)^{n-\ell +1/2}} C^{2\ell} 
&\leq \sum_{\ell=1}^{n-4}  \frac{(2n)^{n+ n-4 +1/2}}{ 2^{2\ell} 4^{n-\ell}} C^{2\ell} \leq \frac{1}{n} \sum_{\ell=1}^{n-4} n^{2n} C^{2\ell} \\
&\leq \frac{n-4}{n C^8} n^{2n}C^{2n} \leq \frac{1}{C^2}n^{2n}C^{2n}, 
 \end{align*}
and the second part is estimated as 
 \[
  \sum_{\ell=n-3}^{n-1}  \frac{(n+\ell)^{n+\ell +1/2}}{ 2^{2\ell} (n-\ell)^{n-\ell +1/2}} C^{2\ell}  \leq \frac{3}{C^2} \frac{(2n)^{2n} C^{2n}}{2^{2n-6}} = \frac{3\cdot 2^6}{C^2} n^{2n} C^{2n}. 
 \]
 Obviously, $2\leq \frac{1}{C^2} \frac{e}{\pi\sqrt{2}} n^{2n} C^{2n}$. 
 Putting everything together, we obtain 
 \[
 \beta_n \leq \frac{194e}{C^2\pi\sqrt{2}} (Cn)^{2n}.  
 \]
 By taking $C>1$ such that $C^2 \geq 194 \frac{e}{\pi\sqrt{2}}$ ($C=11$ suffices) we obtain $\beta_n \leq (Cn)^{2n}$. 
 \end{proof}

\bibliographystyle{amsplain}
\providecommand{\bysame}{\leavevmode\hbox to3em{\hrulefill}\thinspace}
\providecommand{\MR}{\relax\ifhmode\unskip\space\fi MR }
\providecommand{\MRhref}[2]{\href{http://www.ams.org/mathscinet-getitem?mr=#1}{#2}
}
\providecommand{\href}[2]{#2}

\end{document}